\documentclass[12pt]{article}
\usepackage{myclass}

\newcommand*{\QEDD}{\null\nobreak\hfill\ensuremath{\square}}

\newcommand{\Id}{\mathrm{Id}}
\newcommand{\diam}{\mathrm{diam}}

\newcommand{\Homeo}{\mathrm{Homeo}}
\newcommand{\Diff}{\mathrm{Diff}}

\newcommand{\C}{\mathcal{C}}

\newcommand{\R}{\mathbb{R}}
\newcommand{\Z}{\mathbb{Z}}

\newcommand{\Y}{\mathbf{Y}}
\newcommand{\HH}{\mathcal{H}}
\newcommand{\hh}{\mathbf{H}}
\newcommand{\pt}{\prec_\theta}
\newcommand{\st}{\succ_\theta}
\newcommand{\simt}{\sim_\theta}
\newcommand{\dP}{d_{\mathcal{P}}}
\newcommand{\vel}{{\operatorname{vel}}}
\newcommand{\PY}{\mathcal{P}_K^\dagger(\mathbf{Y})}
\newcommand{\CY}{\mathcal{C}^\dagger(\mathbf{Y})}
\newcommand{\scl}{\mathrm{scl}}
\newcommand{\asdim}{\mathrm{asdim}}
\newcommand{\LL}{\mathcal{L}}
\newcommand{\inte}{\mathrm{int}}

\title{\bf Fine projection complex and subsurface homeomorphisms with positive stable commutator length}
\author{\textsc{Yongsheng JIA} \and \sc Yusen LONG}
\date{}

\begin{document}

\maketitle

\begin{abstract}
Drawing inspiration from \cite{bestvina2015constructing}, we construct a family of unbounded quasi-trees for a connected closed oriented surface \(S_g\) of genus \(g\geq 2\), upon which the group \(\Homeo_0(S_g)\) acts coboundedly by isometries. As an application, we show that some surface homeomorphisms preserving a non-sporadic essential subsurface or an essential subsurface homeomorphic to a once-bordered torus can have positive stable commutator length in \(\Homeo_0(S_g)\). Moreover, we provide a version of projection complex that does not require the finiteness condition.
\end{abstract}

\vspace{0.25cm}

\tableofcontents

\section{Introduction}

\epigraph{\it It is our business to puncture gasbags and discover seeds of truth.}{Virginia Woolf,\\ \textit{Thoughts on Peace in an Air Raid (1940)}}

A central theme in geometric group theory is the study of isometric group actions on various non-positively curved spaces or with negative curvature characteristics. Such a methodology has found profound applications in geometric topology, particularly in the understanding of mapping class groups of surfaces. An important example is the curve graph, see for example \cite{masur1999geometry, masur2000geometry}.

Consider a connected closed oriented surface \(S\). In order to give a negative answer to a question asked by Burago, Ivanov and Polterovich (see \cite{BIP2008conjugation}) on whether the group \(\mathrm{Diff}_0(S)\) is \emph{uniformly perfect}, recently Bowden, Hensel and Webb introduced in \cite{bowden2022quasi} a Gromov hyperbolic graph, called the \emph{fine curve graph} (see \autoref{subsec: fine curve graph}) and denoted by \(\C^\dagger(S)\), analogous to the curve graph. Since then, interests grew significantly in actions of homeomorphism groups on this graph (see for example, \cite{leroux2024automorphisms, long2025automorphisms, booth2025automorphisms, choi2025metric, booth2026automorphisms, fanoni2026approximating}), as well as in the geometric and topological properties of this graph (see for example, \cite{fanoni2024curvegraphs, bowden2024boundary, long2025connected, dickmann2026homotopy}). Moreover, much effort has been made to understand the connection between the group actions on these spaces and the dynamics of homeomorphisms on the surface; also see, for example \cite{bowden2022rotationsets, guiheneuf2024parabolic, einabadi2024torusdiff, hensel2025rotation}. Similar strategies are also applied to study the groups of diffeomorphisms on non-orientable surfaces (see for example \cite{kimura2025quasimorphisms, boke2026nonorientablesurfacesstablyunbounded}) and the group of Hamiltonian diffeomorphisms on the 2-sphere \cite{jia2025balanced}. The underlying philosophy is to investigate groups of surface homeomorphisms and diffeomorphisms in a way analogous to the study of mapping class groups.

This paper concerns the group of surface homeomorphisms isotopic to the identity and its actions on several new spaces of non-positive curvature characteristics. 

\subsection{Actions on quasi-trees}
It is known in \cite{bestvina2015constructing} that mapping class groups admit isometric cobounded actions on quasi-trees of infinite diameter. We also remark that there is another more recent version of this construction that uses a sharper version of Behrstock inequality \cite{bestvina2019acylindrical}. In this paper, by a similar approach, we construct a Gromov hyperbolic space called \emph{fine projection complex}, which is also quasi-isometric to a simplicial tree. It turns out that the isometric action of \(\Homeo_0(S)\) on this space is cobounded:

\begin{theorem}\label{thm: cobdd action on quasi tree}
Let \(S\) be a compact connected closed oriented surface. Then the group \(\Homeo_0(S)\) admits cobounded isometric actions on unbounded quasi-trees.
\end{theorem}

\begin{remark}
For any fine projection complex on which \(\Homeo_0(S)\) acts coboundedly by isometries, if this action can be extended to a \(\Homeo(S)\)-action (\emph{e.g.} starting from a \(\Homeo(S)\)-invariant collection of subsurfaces, see \autoref{subsec: 4.A}), then the \(\Homeo(S)\)-action is also cobounded. In contrast, there are fine projection complexes with cobounded \(\Homeo(S)\)-action but the \(\Homeo_0(S)\)-action is not cobounded.
\end{remark}

However, we remark that this space cannot be obtained by directly applying the results in \cite{bestvina2015constructing} for the following reasons. The first reason is that Behrstock's inequality (see \cite{behrstock2006asymptotic} or the \ref{P1} axiom in \cite{bestvina2015constructing}) fails in general for arbitrary collection of subsurfaces. 

Now, consider an essential subsurface \(X\subset S\). As in \cite{bestvina2015constructing}, we wish to build a hyperbolic graph of which the vertices contain the \(\Homeo_0(S)\)-orbit of \(X\), denoted by \(\mathscr{X}\). The finiteness condition \ref{P2} in \cite{bestvina2015constructing} never holds for any collection of subsurfaces containing \(\mathscr{X}\): given \(\theta>0\), if \(X,Y,Z\subset S\) are such that the projection distance (see \eqref{eq: proj dist curve} for definition)
\[d_X^\pi(Y,Z)>\theta,\]
then the image \(X'\) of \(X\) under any small perturbation, which is still in \(\mathscr{X}\), will also verify \(d_{X'}^\pi(Y,Z)>\theta\), and there are uncountably many of them.

In order to deal with these defects, we introduce the notion of \emph{velcrot subsurfaces}, see \autoref{subsec: velcrot subsurfaces}. This notion is inspired by \cite{choi2025metric, hensel2025rotation} where they established a \emph{metric weak proper discontinuity} for some surface homeomorphisms of their actions on the fine curve graph. Roughly speaking, an element \(g\in G\) acting on a metric space enjoys the classical weak proper discontinuity, \emph{abbrv. WPD}, if the joint coarse stabiliser of distant \(g\)-orbit is finite; see \cite{bestvina2002bounded} for the precise definition. However, for similar reasons as above, this finiteness condition never holds for the action of \(\Homeo_0(S)\) on the fine curve graph. Nevertheless, independently in \cite{choi2025metric, hensel2025rotation}, they proposed that such set can be covered by finitely many translates of an arbitrarily small piece around the identity, namely \emph{\(\varepsilon\)-coarse} elements (by adopting the terminology from \cite{choi2025metric}). The first examples of velcrot subsurfaces are a subsurface \(X\subset S\) and its images under \(\varepsilon\)-coarse elements with small \(\varepsilon>0\). Another quintessential example for velcrot subsurfaces is nested isotopic subsurfaces. By introducing velcrotness, we can also show a similar finiteness condition to \ref{P2} axiom in \cite{bestvina2015constructing}, see \autoref{subsec: finiteness}. However, it is worth remarking that velcrotness of subsurfaces is not an equivalent relation and collapsing all velcrot subsurfaces into a ``velcrot class'' will eventually yield the isotopy class of subsurfaces, which is a consequence of \autoref{cor: velcrot implies isotopic}.

Under the mild condition of velcrotness, we are able to prove a weaker version of Behrstock's inequality for the fine setting. The key point is that, upon changing a surface into another one velcrot to it, the projection distance does not change much, see \autoref{cor: velcrot same base}.
\begin{theorem}\label{thm: berhstock main}
There exists \(M>0\) such that the following holds. Let \(X_1,X_2,X_3\) be three essential subsurfaces of \(S\) pairwise intersecting each other essentially. Assume that each \(X_i\) is either non-sporadic or homeomorphic to a once-bordered torus. Suppose in addition that $X_i$ and $X_j$ ($i\neq j$) are either overlapping or isotopic. If \(d^\pi_{X_1}(X_2,X_3)>M\), then
\[d^\pi_{X_2}(X_1,X_3),d^\pi_{X_3}(X_1,X_2)< M.\]
\end{theorem}

Moreover, for a fixed pair of subsurfaces \(X,Z\subset S\), the collection of subsurfaces \(Y\) on which the projection distance \(d^\pi_Y(X,Z)\) is large can be covered by finitely many ``velcrot'' pieces, see \autoref{prop: finiteness} and \autoref{prop: finiteness for torus}. 

These two results above serve as a fine variant of axioms \ref{P1} and \ref{P2} in \cite{bestvina2015constructing}, and allow us to proceed a construction of Bestvina--Bromberg--Fujiwara type.

Two main applications of the Bestvina--Bromberg--Fujiwara construction for mapping class groups are extending quasi-morphisms and computing the asymptotic dimension of mapping class groups \cite{bestvina2015constructing}. For comparison, we will elaborate on these two topics respectively in \autoref{subsec: intro qm} and \autoref{subsec: intro asdim} for \(\Homeo_0(S)\).

\subsection{Extension of quasi-morphisms}\label{subsec: intro qm}

The construction \emph{quasi-morphisms} (for the definition, see \autoref{subsec: qm}) has played a prominent role in various topics in mathematics such as geometric group theory, symplectic geometry, and dynamics, ever since
Gromov’s introduction of bounded cohomology in \cite{gromov1982volume}. Moreover, quasi-morphisms are also closely related to the notion of \emph{stable commutator length} (also see \autoref{subsec: qm}), via \emph{Bavard duality} (\autoref{thm: Bavard Duality}). For a more detailed introduction to stable commutator length, refer to \cite{calegari2009scl}.

When it comes to \(\Diff_0(S)\) or \(\Homeo_0(S)\) for a connected compact oriented surface \(S\), Bowden, Hensel and Webb first constructed unbounded quasi-morphisms on these groups in \cite{bowden2022quasi} using the famous Bestvina--Fujiwara machinery \cite{bestvina2002bounded}. These quasi-morphisms take non-zero values on some homeomorphisms acting loxodromically on the fine curve graph \(\C^\dagger(S)\), and they are homeomorphisms isotopic to a pseudo-Anosov homeomorphism relative to a finite number of points on \(S\) (see \cite[Theorem~1.3]{bowden2022rotation} and \cite[Theorem~A]{guiheneuf2023hyperbolic}). Then these results indicate that some homeomorphisms \(\varphi\in\Homeo_0(S)\) with the closed \(\overline{\mathrm{supp}(\varphi)}=S\) has positive stable commutator length in \(\Homeo_0(S)\). Moreover, in \cite{choi2025metric}, Choi also provided some elements \(\varphi\in\Homeo_0(S)\) with \(\overline{\mathrm{supp}(\varphi)}\subsetneq S\) that has positive stable commutator length in \(\Homeo(S)\).

However, the quasi-morphisms on \(\Homeo_0(S)\) from \cite{bowden2022quasi} take zero value on elements acting elliptically on the fine curve graph, in particular, for \(\varphi\in\Homeo_0(S)\) such that \(\overline{\mathrm{supp}(\varphi)}\) is contained in an essential proper subsurface. So it does not help detecting the positiveness of their stable commutator length on \(\Homeo_0(S)\).

As a positive stable commutator length on the group of homeomorphisms often appears as an indicator of complicated dynamics on the manifold, it is reasonable to guess that complicated dynamics on subsurfaces can also be reflected by a positive stable commutator length on \(\Homeo_0(S)\). Hence, it is natural to ask the following question: \emph{is there any \(\varphi\in\Homeo_0(S)\) that has positive stable commutator length in \(\Homeo_0(S)\), while \(\overline{\mathrm{supp}(\varphi)}\subsetneq S\) is contained in an essential proper subsurface of \(S\)?}

The answer to this question is positive. The solution we provide in this paper can be viewed as extending to \(\Homeo_0(S)\) the Bestvina--Fujiwara type quasi-morphisms defined on the subgroup \(\Homeo_0(X;\partial X)<\Homeo_0(S)\), the group of homeomorphisms on a subsurface \(X\) fixing pointwise the boundary \(\partial X\), via the Bestvina--Bromberg--Fujiwara type construction. This will further yield a positive stable commutator length on \(\Homeo_0(S)\) for elements in \(\Homeo_0(X;\partial X)\) after the Bavard duality. 

For curiosity, we mention that \cite{bestvina2016scl} gives a characterisation of elements with positive stable commutator length in the mapping class group of a compact surface. For non-compact surfaces with non-displaceable compact subsurface, the Bestvina--Bromberg--Fujiwara type construction also yields various quasi-morphisms taking non-zero values on mapping classes that preserves a subsurface, see \cite{domat2022big, horbez2022big}.

Extensions of quasi-morphisms usually require certain hyperbolic characteristics. For more information about extending subgroup quasi-morphisms to the entire group, see for example an incomplete list of references on extension of quasi-morphisms: \cite{hull2013induced, kawasaki2024survey, tao2025extensiontheoremquasimorphisms}.

Modifying the fine projection complex in \autoref{thm: cobdd action on quasi tree} into a \emph{blown-up fine projection complex} and investigating the action of \(\Homeo_0(S)\) on it, we are able to show the following:

\begin{theorem}\label{thm: qm version}
Let \(S\) be a connected closed oriented surface of genus at least two and \(G\) be \(\Homeo_0(S)\). Then for every essential subsurface \(\Sigma\subset S\) that is either non-sporadic or a once-bordered torus, there exists an unbounded (\(C^0\)-continuous) homogeneous quasi-morphism \(\varphi\colon G\to \R\) and an element \(g\in G\) such that \(g(\Sigma)=\Sigma\) and \(\varphi(g)\neq 0\).
\end{theorem}

The following theorem can be viewed as a corollary of \autoref{thm: qm version}, from which it follows combining with \autoref{thm: non coarsely invertible} and \autoref{rem: existence of non coarsely invertible}:

\begin{theorem}\label{thm: main}
Let \(S\) be a connected closed oriented surface of genus at least two and \(G\) be \(\Homeo_0(S)\). Then for every essential subsurface \(\Sigma\subset S\) that is either non-sporadic or a once-bordered torus, there exists an element \(g\in G\) such that \(g(\Sigma)=\Sigma\) and \(\scl_G(g)>0\).
\end{theorem}

The detailed construction of blown-up fine projection complex is given in \autoref{sec-5}. This graph is a \(\delta\)-hyperbolic (see \autoref{thm: CY is hyperbolic}) and we show this by using the \emph{Guessing Geodesics Lemma} (see \autoref{prop: ggl}). This is a powerful tool for proving the Gromov hyperbolicity of a metric space, and the readers can refer to \cite{hamenstadt2007geometry, masur2013geometry, bowditch2014uniform} for further details. Examples of the application of the Guessing Geodesics Lemma can be also found in \cite{hensel2015slim, przytycki2017note, jia2025balanced}.

For the construction of the blown-up fine projection complex, roughly speaking, we start from a \(\Homeo_0(S)\)-invariant collection \(\Y\) of subsurfaces that are pairwise overlapping or isotopic and that are either non-sporadic or homeomorphic to a once-bordered torus, then by taking the union of fine curve graph \(\C^\dagger(X)\) of any subsurface \(X\in \Y\) and connecting them in a way that they are quasi-isometrically embedded as subgraphs, we can build an unbounded Gromov hyperbolic graph \(\CY\). For velcrot subsurfaces \(X,Z\in \Y\), the images of \(\C^\dagger(X)\) and \(\C^\dagger(Z)\) in \(\CY\) have bounded Hausdorff distance. Moreover, two elements \(f,g\in \Homeo_0(X;\partial X)\) acting by independent loxodromic isometries on \(\C^\dagger(X)\) will also act by independent loxodromic isometries on \(\CY\), see \autoref{prop: ind loxo}. The construction we provide herein is also similar to the Bestvina--Bromberg--Fujiwara construction, whereas we need to deal with the technicalities brought forth by velcrotness.

Although there is a way to axiomatise our construction of the fine projection complex and the blown-up fine projection complex as in \cite{bestvina2015constructing}, the axioms that we will need are numerous and it seems to us that there is no obvious and canonical way to check the unboundedness of the fine projection complex. For the reason of keeping the geometric intuition, we do not offer an axiomatisation in the current article.

\subsection{Asymptotic dimension and the Gromov boundary}\label{subsec: intro asdim}

An other principal result in \cite{bestvina2015constructing} is that they are able to embed the mapping class group of the surface into a finite product of blown-up projection complexes and show that the \emph{asymptotic dimension} of mapping class groups is finite. This notion was first introduced by Gromov in \cite{gromov1993asymptotic} as an asymptotic invariant for finitely generated groups. The asymptotic dimension can be defined for general metric spaces and is a quasi-isometric invariance, see, for example, \cite{bell2008asymptotic}; in connection to the coarse setting, see \cite{roe2003lectures}. Hence, one can talk about the asymptotic dimension of a group with well-defined quasi-isometry type, \emph{e.g.} finitely generated groups and compactly generated groups.

However, one can discuss the asymptotic dimension for topological groups far beyond the locally compact setting. For Polish groups, \emph{i.e.} completely metrisable and separable groups, there is also a well established theory for their large-scale geometry \cite{rosendal2021coarse}. Similarly to the locally compact cases, there is a class of Polish groups admitting well-defined quasi-isometry type, namely the \emph{coarsely bounded generated} Polish groups, or \emph{CB generated} Polish groups for abbreviation. These groups can be equipped with a class of maximal left-invariant compatible metrics \cite[Theorem~1.2]{rosendal2021coarse}, and by maximality, these metrics are quasi-isometric to each other. This implies that CB generated groups have well-defined quasi-isometry type and thus well-defined asymptotic dimension. In particular, Mann and Rosendal showed that for a compact surface \(S\), the group \(\Homeo_0(S)\) equipped with the compact-open topology, which renders the group Polish, is CB generated \cite{mann2018large}.

In \cite{bestvina2015constructing}, by embedding the mapping class groups quasi-isometrically into a finite product of blown-up projection complexes, which have a finite asymptotic dimension, Bestvina, Bromberg and Fujiwara showed that the asymptotic dimension of mapping class groups are finite. Nevertheless, this strategy fails drastically in the fine setting: unlike mapping class groups, homeomorphism groups \(\Homeo_0(S)\) have an infinite asymptotic dimension, as separable metric spaces of arbitrarily large asymptotic dimension can be coarsely embedded into \(\Homeo_0(S)\) \cite[Proposition~20]{mann2018large}. Moreover, the product of blown-up fine projection complexes that we constructed in this article is also insufficient to encode the large-scale geometry of \(\Homeo_0(S)\). Indeed, the group \(\Homeo_\Z(\R)\), which is quasi-isometric to \(\Z\), can be coarsely embedded into \(\Homeo_0(S)\) via orientation-preserving homeomorphisms on the interior of an essential annulus \cite[Proposition~40]{mann2018large}, but its action on any fine projection complex that we construct in this paper is elliptic. 

Furthermore, the blown-up fine projection complexes have infinite asymptotic dimension, see \autoref{cor: CY infinite asdim}. This is a corollary of the following result, as one can embed the fine curve graph \(\C^\dagger(X)\) of a non-sporadic subsurface \(X\) quasi-isometrically into \(\CY\) (see \autoref{prop: fine curve graph bi-lip}):

\begin{theorem}\label{thm: asdim}
Let \(S\) be a compact connected oriented non-sporadic surface. Then the asymptotic dimension of \(\C^\dagger(S)\) is infinite.
\end{theorem}

This topic is somehow related to the topology of the Gromov boundary of \(\C^\dagger(S)\). We prove \autoref{thm: asdim} by homeomorphically embedding a compact subset in \(\partial \C^\dagger(S)\) of arbitrarily large topological dimension (\autoref{prop: embed}). As the asymptotic dimension of a Gromov hyperbolic space is bounded below by the topological dimension of compact subset on its Gromov boundary (\autoref{prop: bnd dim}), this allows us to conclude \autoref{thm: asdim}.

Another example of spaces that can contain completely metrisable subspaces with any topological dimension is the infinite-dimensional separable Hilbert space \(\ell^2\) and spaces that is locally homeomorphic to \(\ell^2\), \emph{i.e.} an \emph{\(\ell^2\)-manifold}. Among them \(\Homeo(S)\) is itself an \(\ell^2\)-manifold \cite[Théorème~7.1.3]{leroux1997etude}. As \(\Homeo(S)\) acts continuously and minimally on the Gromov boundary \(\partial\C^\dagger(S)\) \cite[Proposition~6.4]{long2025connected}, it is also natural to ask:
\begin{question}
Given a surface \(S\) that is either non-sporadic or a once-bordered torus, is the Gromov boundary of \(\C^\dagger(S)\) an \(\ell^2\)-manifold?
\end{question}

\subsection*{Acknowledgement} The authors wish to thank Indira Chatterji, Inhyeok Choi, Federica Fanoni, Koji Fujiwara, Pierre-Antoine Guihéneuf, Sebastian Hensel, Frédéric Le Roux, Dong Tan, Robert Tang, Bingxue Tao, Richard Webb, and Wenyuan Yang for their stimulating questions and helpful discussions. The authors also thank Mladen Bestvina, Francesco Fournier-Facio and Alessandro Sisto for valuable comments.

This work was partially carried out during the visits of the second author in the University of Manchester and BICMR, Peking University. The second author wishes to express his gratitude for their hospitality.

The first author is supported by EPSRC DTP EP/V520299/1. The second author acknowledges the support from the ANR project Grant GALS (ANR-23-CE40-0001).
\section{Preliminaries}

This section aims at giving the definition of some basic notions and reviewing several known results on Gromov hyperbolic space, quasi-morphisms and stable commutator length, Bestvina--Fujiwara machinery, as well as fine curve graphs and surface homeomorphisms.

\subsection{Gromov hyperbolicity}\label{S2.A} Some classical references for Gromov hyperbolicity are \cite{bridson2013metric, das2017geometry}. Here we will briefly recall the definition of a geodesic metric space satisfying this condition, and provide several properties of the isometries on this space.

Let \((X,d)\) be a geodesic metric space. We denote by \([x,y]\) any geodesic segment between two points \(x,y\in X\). We say that \(X\) is {\it\(\delta\)-hyperbolic} for some \(\delta\geq 0\) if for any \(x,y,z\in X\), we have
\[\mathcal{N}_\delta([x,y]\cup [y,z])\supset [x,z]\,,\]
where \(\mathcal{N}_\delta(A)\) stands for the \(\delta\)-neighbourhood of the subset \(A\subset X\)

We will now introduce a remarkable and very powerful criterion for proving Gromov hyperbolicity, namely the \emph{Guessing Geodesics Lemma}. This was originally discovered by Masur and Schleimer, with an alternative proof by Bowditch, see \cite{masur2013geometry,bowditch2014uniform}.

\begin{proposition}[Guessing geodesics lemma]\label{prop: ggl}
Let \((\Gamma,d_\Gamma)\) be a graph. If there exists some constant \(\lambda> 0\) such that for any \(x,y\in \Gamma\), we can find a connected subgraph \(\mathcal{L}(x,y)\subset \Gamma\) containing \(x,y\) and satisfying the following conditions:
\begin{enumerate}[label=(G\arabic*),topsep=0pt, itemsep=-1ex, partopsep=1ex, parsep=1ex ]
    \item for any \(x,y,z\in \Gamma\), \(\mathcal{L}(y,z)\) is contained in the \(\lambda\)-neighbourhood of \(\mathcal{L}(x,y)\cup \mathcal{L}(y,z)\);
    \item for any \(x,y\in \Gamma\) with \(d_\Gamma(x,y)\leq 1\), the diameter of \(\mathcal{L}(x,y)\) in \(\Gamma\) is at most \(\lambda\);
\end{enumerate}
Then \(\Gamma\) is \(\delta\)-hyperbolic for some \(\delta>0\) depending only on \(\lambda\).
\end{proposition}
\begin{remark}
The connected subgraph \(\mathcal{L}(x,y)\) associated to \(x,y\in \Gamma\) is called the \emph{guessing geodesic} between \(x\) and \(y\), and is known to have uniformly bounded Hausdorff distance from a geodesic between \(x\) and \(y\) (bounded in terms of \(\lambda\)). In fact, \(\delta\) can be chosen as any number greater than or equal to \((3m - 10\lambda)/2\), where \(m\) is any positive real number satisfying \(2\lambda (6 + \log_2(m + 2)) \leq m\).
\end{remark}

Although the above version for graphs is sufficient for our purpose, we remark for curiosity that there is a more general version of this result in \cite[Proposition 3.5]{hamenstadt2007geometry}, which is stated for geodesic metric spaces \(X\) with a continuous path \(\eta_{xy}\colon [0,1]\to X\) connecting any pair \(x,y\in X\) and satisfying similar conditions as above.

Recall that an {\it isometry} \(f\colon X\to X\) is a map such that \(d_X\big(f(x),f(y)\big)=d_X(x,y)\) for all \(x,y\in X\). We define its {\it asymptotic translation length} as
\[|f|_X\coloneqq \lim_{n\to +\infty} \frac{1}{n}d_X\big(f^n(x),x\big).\]
For a \(\delta\)-hyperbolic space, we have the following classification of isometries according to their asymptotic translation length \cite[\S 8]{gromov1987hyperbolic}:
\begin{definition}
Let \(X\) be a \(\delta\)-hyperbolic space. An isometry \(f\colon X\to X\) is called
\begin{enumerate}[label=(\arabic*), topsep=0pt, itemsep=-1ex, partopsep=1ex, parsep=1ex ]
    \item \emph{elliptic}, if \(f\) has bounded orbits;
    \item \emph{parabolic}, if \(|f|_X=0\) and has no bounded orbit;
    \item \emph{loxodromic}, or also \emph{hyperbolic}, if \(|f|_X>0\).
\end{enumerate}
\end{definition}

Recall that a map \(f\colon (X,d_X)\to (Y,d_Y)\) between two metric spaces is a {\it \((\lambda,k)\)-quasi-isometric embedding} if there exist \(\lambda\geq 1\) and \(k>0\) such that
\begin{align}\label{eq:qi}
\frac{1}{\lambda}d_X(x,y)-k\leq d_Y\big(f(x),f(y)\big)\leq \lambda d_X(x,y)+k
\end{align}
for every \(x,y\in X\). In particular, we say that the map \(f\) is {\it \(\lambda\)-bi-Lipschitz} if \(k=0\) in \eqref{eq:qi}. If in addition, the map \(f\) is {\it essentially surjective}, {\it i.e.} there exists \(C>0\) such that \(d_Y\big(y,f(X)\big)<C<\infty\) for all \(y\in Y\), then \(f\) is a quasi-isometry and \(X\) and \(Y\) are said to be quasi-isometric. We remark that being quasi-isometric is an equivalence relation. A {\it quasi-geodesic} in a metric space \(X\) is a quasi-isometric embedding of a real interval into \(X\). We remark that for geodesic metric spaces, Gromov hyperbolicity is a quasi-isometry invariant, see for example \cite[Theorem~III.H.1.9]{bridson2013metric}.

If \(X\) is a \(\delta\)-hyperbolic space and \(g\) is a loxodromic isometry of \(X\), then it admits a bi-infinite quasi-geodesic that is \(g\)-invariant, which we will later refer to it as a \emph{quasi-axis} of \(g\). It is always convenient to blur the distinction between quasi-axis and its image. By Morse Lemma, for any parameter \((K,L)\), there exists a constant \(B\coloneqq B(K,L,\delta)\) such that any two \((K,L)\)-quasi-axis of \(g\) stay within the \(B\)-neighbourhood of each other; see, for example, \cite[Theorem~III.H.1.7]{bridson2013metric}.

For a geodesic Gromov hyperbolic space \(X\), one can also define its \emph{Gromov boundary} \(\partial X\) by the equivalence classes of quasi-geodesic rays issued from a base point \(o\in X\), and we say that two quasi-geodesic rays are equivalent if they are within a finite Hausdorff distance. Moreover, the boundary \(\partial X\) does not depend on the choice of the base point \(o\in X\). Also, it carries a complete metric called \emph{visual distance} that also defines its topology, making \(\partial X\) a completely metrisable space. See \cite[\S 3.6]{das2017geometry} and \cite{vasala2005gromov} for detailed treatment.

\subsection{Quasi-morphisms and stable commutator length}\label{subsec: qm}
Let us first review the following definition:

\begin{definition}[Quasi-morphism]
A \emph{quasi-morphism} on a group $G$ is a map $\mu:G\rightarrow \mathbb{R}$ such that there is a least constant $D(\mu)\geq 0$, depending only on $\mu$, called the \emph{defect} of $\mu$, with the property that $$|\mu(gh)-\mu(g)-\mu(h)|\leq D(\mu), \text{ for all } g,h \in G.$$
\end{definition}

The set of all quasi-morphisms on a fixed group $G$ is easily seen to be a (real) vector space; we denote this vector space by $\widehat{Q}(G)$. Respectively by $C_{b}^{1}(G;\mathbb{R})$ and $\operatorname{Hom}(G;\R)=H^1(G;\mathbb{R})$ we denote the subspaces of $\widehat{Q}(G)$ consisting of (real-valued) bounded functions and of homomorphisms. 

A quasi-morphism $\mu$ is said to be \emph{homogeneous} if $\mu(g^{n})=n\mu(g)$ for any $g\in G$ and $n\in \mathbb{Z}$. We denote by $Q(G)$ the subspace of $\widehat{Q}(G)$ consisting of homogeneous quasi-morphisms. For any quasi-morphism $\mu\in \widehat{Q}(G)$, one can obtain a homogeneous quasi-morphism $\widetilde{\mu}\in Q(G)$, called the \emph{homogenisation} of $\mu$, as follows: 
$$
\widetilde{\mu}(g)\coloneqq \lim_{n\rightarrow \infty} \frac{1}{n} \mu(g^n).
$$ 
This limit always exists for each element $g$ of $G$ since the sequence $\big(\mu(g^n)\big)_n$ is sub-additive with bounded error. Moreover, for any $g\in G$, we have $|\widetilde{\mu}(g)-\mu(g) |\leq D(\mu)$, see for example \cite{bavard1991longueur}. In other words, a quasi-morphism $\mu$ is (uniquely) written as the sum of a homogeneous quasi-morphism $\widetilde{\mu}$ and a bounded function. By definition, bounded functions on groups are quasi-morphisms. We can then identify the quotient space $\widehat{Q}(G) / C_{b}^{1}(G;\mathbb{R})$ with $Q(G)$.

Note that $C_{b}^{1}(G;\mathbb{R})\cap \operatorname{Hom}(G;\R)=0$. We are then interested in the quotient spaces 
$$
Q(G)=\widehat{Q}(G)/C_{b}^{1}(G;\mathbb{R})
$$
and 
$$
\widetilde{Q}(G)=\widehat{Q}(G)/(C_{b}^{1}(G;\mathbb{R})+\operatorname{Hom}(G;\R))\simeq Q(G)/H^1(G;\mathbb{R})
$$
as any homogeneous quasi-morphism $\mu$ is invariant under conjugations, \emph{i.e.} for all \(a,b\in G\), 
$$ \mu(aba^{-1})=\mu(b)\,. $$ 

We now introduce an algebraic dual of the quasi-morphism, which is called the \emph{stable commutator length} that we will introduce in the following.
\begin{definition}[Commutator length]
A \emph{commutator} is an element in $G$ that can be expressed in the form $[a,b]=aba^{-1}b^{-1}$ for some $a,b\in G$. The subgroup generated by commutators is called the \emph{commutator subgroup} and is denoted by \([G,G]\). The \emph{commutator length} of an element $g\in G$, denoted by $\operatorname{cl}_G(g)$, is defined to be the word length of $g$ with respect to the set of all commutators:
$$
\operatorname{cl}_G(g)\coloneqq\inf \Big\{n\in \mathbb{N}_{\geq 0}\mid g= \prod_{i=1}^n [a_i,b_i];a_i,b_i\in G\Big\}\in \mathbb{N}_{\geq 0}\cup \{\infty\},
$$
where we allow $\operatorname{cl}_G(g)=\infty$. Note that $\operatorname{cl}_G(g)<\infty$ if and only if $g\in [G,G]$.
\end{definition}

\begin{definition}[Stable commutator length]
For $g\in [G,G]$, the \emph{stable commutator length} of $g$, denoted by $\operatorname{scl}_G(g)$ is the following limit:
\begin{equation} \label{eq: scl}
\operatorname{scl}_G(g)\coloneqq\lim_{n\to \infty} \frac{\operatorname{cl}_G(g^n)}{n}.
\end{equation}
\end{definition}
\begin{remark}
For each fixed $g\in [G,G]$, the function $n\mapsto \operatorname{cl}_G(g^n)$ is non-negative and sub-additive. Hence, the limit in \eqref{eq: scl} exists. However, we can further define the stable commutator length for general elements in \(G\). If $g\notin [G,G]$ but admits a power $g^k\in [G,G]$ for some \(k>0\), then we define $\operatorname{scl}_G(g)\coloneqq\operatorname{scl}_G(g^k)/k$, and by convention define $\operatorname{scl}_G(g)=\infty$ if no positive power of $g$ lies in $[G,G]$.
\end{remark}

The following result is called the \emph{Bavard duality} and indicates how stable commutator length is related to quasi-morphisms. See, for example, \cite{bavard1991longueur} or \cite[Theorem~2.70]{calegari2009scl} for detailed proof.
\begin{theorem}[Bavard Duality]\label{thm: Bavard Duality}
Let $G$ be a group. Then for any $g\in [G,G]$, we have the following equality:
$$
\operatorname{scl}_G(g)= \sup_{\varphi\in \widetilde{Q}(G)} \frac{|\varphi(g)|}{2D(\varphi)}\ .
$$
\end{theorem}

\subsection{Bestvina--Fujiwara quasi-morphisms} In the following, we will introduce the famous Bestvina--Fujiwara machinery from \cite{bestvina2002bounded}. For our convenience, we will restate some definitions and results from \cite{bestvina2002bounded}. However, readers should note that the terminology that we use here is different from the original work \cite{bestvina2002bounded}.

Recall that given a loxodromic element \(g\) acting on a \(\delta\)-hyperbolic space \(X\), there is a constant \(B\coloneqq B(K,L,\delta)\) such that any two \((K,L)\)-quasi-axis of \(g\) stay within the \(B\)-neighbourhood of each other.

\begin{definition}[Independent loxodromics]
Let \(X\) be a \(\delta\)-hyperbolic graph and \(G\) be a group acting on it by isometries. Let us consider two loxodromic elements \(f,g\in G\) with respective quasi-axes \(A_f,A_g\). We say that they are \emph{independent}, denoted by \(f\nsim g\), if for any \(B\geq B(K,L,\delta)\), there exists a segment \(J\subset A_g\) such that for any \(h\in G\), the translate \(hJ\) is not contained in the \(B\)-neighbourhood of \(A_f\). Otherwise, we say that they are dependent and we write \(f\sim g\).
\end{definition}

We remark that if \(f\nsim g\), then by taking \(h=\Id\), we can conclude that \(A_f\) and \(A_g\) are not within a finite Hausdorff distance. This implies that the existence of two independent loxodromic elements ensures that the \(G\)-action on \(X\) is \emph{of general type} (or also \emph{non-elementary} in literature); see, for example, \cite{caprace2015amenable}.

\begin{theorem}[Theorem 1, \cite{bestvina2002bounded}]\label{Bestvina--Fujiwara machinery}
Suppose that \(G\) acts on a \(\delta\)-hyperbolic graph by isometries and the action is of general type. Suppose also that there exist independent loxodromic elements \(f\nsim g\). Then the space $\widetilde{Q}(G)$ is infinite-dimensional.
\end{theorem}
All the quasi-morphisms in \autoref{Bestvina--Fujiwara machinery} can be constructed explicitly. In the remanent of this subsection, we will briefly recall their construction. The model case of the free group is due to Brooks~\cite{brooks1980some}.

Let \(w\) be a finite (oriented) path in \(X\). Let \(|w|\) denote the length of \(w\). For \(g\in G\), we denote by the composition \(g\circ w\) a copy of \(w\) by \(g\)-translation. It is clear that \(|g\circ w|=|w|\). Let \(\alpha\) be a finite path. We define
$$
|\alpha|_{w}=\{\text{the maximal number of non-overlapping copies of $w$ in $\alpha$}\}.
$$
Suppose that \(x,y\in X\) are two vertices and that \(R\) is an integer with \(0<W<|w|\). We define the integer 
$$
c_{w,W}(x,y)=d(x,y)-\inf_{\alpha}(|\alpha|-W|\alpha|_{w}),
$$
where \(\alpha\) ranges over all paths from \(x\) to \(y\). Fixing a base point \(x_{0}\in X\), we define \(h_{w}\colon G\rightarrow \mathbb{R}\) by
$$
h_{w}(g)= c_{w,W}(x_{0},g(x_{0}))-c_{w^{-1},W}(x_{0},g(x_{0})),
$$
which is a quasi-morphism defined on \(G\) but not necessarily homogeneous.

The following theorem provides a sufficient condition for an element \(g\in G\) to have positive stable commutator length.
\begin{theorem}[Proposition~5, \cite{bestvina2002bounded}]\label{thm: non coarsely invertible}
Let \(X\) be a \(\delta\)-hyperbolic space and \(G\) be a group acting on \(X\) by isometries. If \(f\in G\) is a loxodromic element such that \(f\nsim f^{-1}\), then there is a homogeneous quasi-morphism that is unbounded on the group generated by \(f\), and \emph{a fortiori} taking non-zero values on \(f\). In particular, \(f\) has positive stable commutator length.
\end{theorem}

\begin{remark}\label{rem: existence of non coarsely invertible}
We say that a loxodromic element \(f\in G\) is \emph{quasi-invertible} if \(f\sim f^{-1}\). From \cite[Proposition~2]{bestvina2002bounded}, we can see that if \(G\) acts on a Gromov hyperbolic graph \(X\) by isometries and the action is of general type with two independent loxodromic elements \(g_1,g_2\in G\), then there are infinitely many loxodromic elements in the subgroup of \(\langle g_1,g_2\rangle<G\) that are not quasi-invertible. In particular, the non-quasi-invertible element \(f\) in \autoref{thm: non coarsely invertible} can be chosen in \(\langle g_1, g_2\rangle\).
\end{remark}

\subsection{Fine curve graph}\label{subsec: fine curve graph} Let \(S=S_{g,b}\) be a connected closed oriented surface of finite type of genus \(g \geq 2\), with \(b\) boundary components. Let \(\xi(S)=3g+b-3\) be the complexity of \(S\). We denote by \(\partial S\) the boundary of $S$. We say that a surface is \emph{non-sporadic} if \(\xi(S)\geq 2\).

An {\it essential simple closed curve} \(\gamma\) on \(S\) is a proper \(C^0\)-embedding of the circle, \(\gamma: S^1 \hookrightarrow S\), with the property that \(\gamma\) does not bound a disc nor a boundary component. In the sequel, for most of the time, we will identify a curve with its image.

In \cite{bowden2022quasi}, Bowden, Hensel and Webb introduce an analogue to curve graph for the group of homeomorphisms on \(S\), namely, the {\it fine curve graph}, denoted by \(\mathcal{C}^\dagger(S)\). It is a graph whose vertices correspond to essential simple closed curves on \(S\) and two vertices in this graph are connected by an edge if the corresponding curves are disjoint in \(S\). Endowing \(\mathcal{C}^\dagger(S)\) with the simplicial distance \(d^\dagger\), there exists a \(\delta>0\), independent of the complexity of the surface, such that the fine curve graph \(\mathcal{C}^\dagger(S)\) is \(\delta\)-hyperbolic \cite[Theorem 3.8]{bowden2022quasi}. The idea is to approximate the fine curve graph by the \emph{surviving curve graph} of finitely punctured surfaces, and the uniformity of the hyperbolic constant \(\delta\) find its root in the uniform hyperbolicity of the non-separating curve graphs \cite{rasmussen2020uniform}.

To be more precise, for a finite subset \(P\subset S\), let us denote by \(\C^s(S\setminus P)\) the surviving curve graph on \(S\setminus P\). The vertices of this graph are the isotopy classes of essential simple closed curves on \(S\setminus P\) that remain essential on \(S\) and we connect two vertices if the corresponding two curves admit disjoint representatives. Now, we have the following:
\begin{lemma}[Lemma 3.4, \cite{bowden2022quasi}]\label{lem: dagger}
Let \(\alpha,\beta\in \C^\dagger(S)\) be two transverse curves. Then for a finite subset \(P\subset S\) such that \(\alpha,\beta\) are in minimal position on \(S\setminus P\), we have
\[d^\dagger(\alpha,\beta)=d_{\C^s(S\setminus P)}([\alpha]_{S\setminus P},[\beta]_{S\setminus P}).\]
\end{lemma}
\begin{remark}
Although \cite[Lemma 3.4]{bowden2022quasi} is originally stated for non-sporadic closed surfaces, the arguments therein also hold for torus and surfaces with boundary components.
\end{remark}

Similarly to the classical case of curve graphs, on the fine curve graph, the distance between two transverse curves is also controlled by their intersection number.
\begin{proposition}[Proposition 3.8, \cite{long2025connected}]\label{prop: intersection}
Let \(S\) be above and let \(\alpha,\beta\in \mathcal{C}^\dagger(S)\) be two transverse curves on \(S\). Then \(d^\dagger(\alpha,\beta)\leq 2|\alpha\cap \beta|+1\).
\end{proposition}

Let \(Y\) be an essential non-sporadic proper subsurface of \(S\), \emph{i.e.} \(\partial Y\) is a finite collection of essential curves on \(S\). Let \(\alpha \in \mathcal{C}^\dagger(S)\) be a simple closed curve on \(S\). We say that \(\alpha\) intersects \(Y\) {\it essentially} if at least one connected component of \(\alpha \cap Y\) is an essential arc or curve in \(Y\), {\it i.e.}, not isotopically trivial in \(Y\). It is worth remarking that a general curve \(\alpha\) may not intersect \(\partial Y\) transversely and this notion of essential intersection may not be invariant under isotopy.

The subsurface projection of a curve on \(S\) that is essentially intersecting \(Y\) can be defined as follows:
\begin{definition}[Subsurface projection]\label{def: subsurface proj}
Suppose that $S$ and $Y$ are given as above. Let $\mathcal{C}^\dagger(Y)$ be the fine curve graph for $Y$ and let $\mathcal{P}\left(\mathcal{C}^\dagger(Y)\right)$ be its power set. We define a map $\pi_Y\colon \mathcal{C}^\dagger(S)\to\mathcal{P}\left(\mathcal{C}^\dagger(Y)\right)$ in the following way: for each $\alpha\in \mathcal{C}^\dagger(S)$, the image $\pi_Y(\alpha)$ is defined as
\begin{itemize}[topsep=0pt, itemsep=-1ex, partopsep=1ex, parsep=1ex ]
    \item $\{\alpha\}$ if $\alpha\subset Y$;
    \item $\emptyset$ if $\alpha$ does not intersect $Y$ essentially;
    \item all essential curves in \(\C^\dagger(Y)\) that is disjoint from one essential subarc of \(\alpha\cap Y\), if $a\not\subset Y$ but $\alpha$ intersects $Y$ essentially.
\end{itemize}
The map $\pi_Y$ is called the {\it subsurface projection} of $\alpha$ on $Y$.
\end{definition}

Moreover, this projection is in fact coarsely well-defined in the following sense:
\begin{proposition}[Proposition 4.2, \cite{long2025connected}]\label{prop: bdd diam}
With the simplicial distance $\big(\mathcal{C}^\dagger(Y),d^\dagger_Y\big)$, for any curve $\alpha \in \mathcal{C}^\dagger(S)$ intersecting \(Y\) essentially, the diameter of the set $\pi_Y(\alpha)$ is bounded in $\mathcal{C}^\dagger(Y)$, {\it i.e.}, $\diam_{d^\dagger_Y}(\pi_Y(\alpha)) \leq 12.$
\end{proposition}

Moreover, we have the following distance estimation:
\begin{proposition}[Proposition 4.3, \cite{long2025connected}]\label{prop_disjoint_subsurface_proj}
    Let $S$, $Y$, $d^\dagger_Y$, and $\pi_Y$ be as defined above. For any two transverse curves $\alpha, \beta \in \mathcal{C}^\dagger(S)$ that intersect the subsurface $Y$ essentially. Suppose in addition that there are essential arcs \(\alpha'\subset \alpha\cap Y\) and \(\beta'\subset \beta\cap Y\) such that $|\alpha' \cap \beta'| = p \geq 0$. Then for any $\gamma_\alpha \in \pi_Y(\alpha)$ and $\gamma_\beta \in \pi_Y(\beta)$, the following inequality holds:
    $$d_Y^\dagger\big(\gamma_\alpha, \gamma_\beta\big) \leq 8p + 12.$$
\end{proposition}

Finally, similar to the case of curve graphs, we have the following version of the bounded geodesic image theorem for fine curve graphs:
\begin{theorem}[Theorem 1.3, \cite{long2025connected}]\label{thm: bgit}
Given a surface \(S\) with \(g(S) \geq 2\), there exists a constant \(M>0\) such that whenever \(Y\) is an essential subsurface that is either non-sporadic or homeomorphic to a once-bordered torus, given any \(g = (\gamma_i)\) geodesic path in \(\mathcal{C}^\dagger(S)\) such that \(\gamma_i\) intersects \(Y\) essentially for all \(i\), then \(\diam_{d_Y^\dagger} \big(\pi_Y(g)\big) \leq M\).
\end{theorem}

\begin{remark}
The proof of \autoref{thm: bgit} also applies to subsurfaces that are homeomorphic to a once-bordered torus, as the principal ingredients of the proof are the uniform hyperbolicity of the surviving curve graph (see for example \cite[\S 5.2]{bowden2022quasi}) and the upper bound of distance in the fine curve graph of the subsurface by the intersection numbers, which holds in once-bordered tori, see \autoref{sec-6}.
\end{remark}
\section{Surgery on surface}\label{sec: surf surg}
The Bestvina--Bromberg--Fujiwara construction of projection complex requires us to provide a collection \(\Y\) of metric spaces with coarsely well-defined projection to each other so that if the projections of \(Y, Z\) to \(X\) are far away, then the projections of \(X,Y\) to \(Z\), as well as the projections of \(X,Z\) to \(Y\), are close, and that for fixed \(Y,Z\), there are only finitely many \(X\) to which the projections of \(Y, Z\) are far away, see \cite{bestvina2015constructing}. However, for the fine curve graph of essential subsurfaces (not the isotopy classes of them), these conditions do not hold. This section aims at introducing the notion of velcrot subsurfaces to establish weaker conditions that allow us to run a similar machinery.

\subsection{Velcrot subsurfaces}\label{subsec: velcrot subsurfaces}
The following notion is inspired by the metric WPD properties introduced in \cite{choi2025metric, hensel2025rotation}, but the definition of finiteness condition in their metric WPD properties relies on some \(\C^0\)-distance. Here, we offer a topological analogue for subsurfaces.

\begin{definition}
Let \(X,Y\subset S\) be two non-sporadic essential subsurfaces. We say that \(X\) and \(Y\) are \emph{velcrot} if $\C^{\dagger}(X)\cap \C^{\dagger}(Y)$ is of infinite diameter in both $\C^{\dagger}(X)$ and $\C^{\dagger}(Y)$.
\end{definition}
\begin{remark}
In particular, $\C^{\dagger}(X)\cap \C^{\dagger}(Y)$ is a subset of both $\pi_{X}(\C^{\dagger}(Y))$ and $\pi_{Y}(\C^{\dagger}(X))$. Hence, if $X$ and $Y$ are velcrot, then both $\pi_{X}(\C^{\dagger}(Y))$ and $\pi_{Y}(\C^{\dagger}(X))$ are unbounded.
\end{remark}

This definition is purely metric geometric. It might be useful for settings other than subsurfaces on a surface. However, under the setting of essential subsurfaces and their fine curve graphs, the definition of velcrotness can also be given in a purely topological way, as we will explain below.

\begin{lemma}\label{lem:velcrot not intersec}
If \(X,X'\subset S\) are velcrot, then \(\partial X\) does not intersect \(X'\) essentially and \emph{vice versa}.
\end{lemma}
\begin{proof}
We may assume for contradiction that \(\partial X\) intersect \(X'\) essentially, then for any \(\alpha\in \C^\dagger(X)\cap \C^\dagger(X')\), the essential arc \(\partial X\cap X'\) and the curve \(\alpha\) are disjoint, so by \autoref{prop_disjoint_subsurface_proj}, we have \(d^\dagger_{X'}\big(\alpha,\beta\big)\leq12\) for any \(\beta\in \pi_{X'}(\partial X)\). But \autoref{prop: bdd diam} and \autoref{prop_disjoint_subsurface_proj} imply that the diameter of \(\pi_{X'}(\partial X)\) is bounded in \(\C^\dagger(X')\), which forces that \(\alpha\) is contained in a bounded subset in \(\C^\dagger(X')\), contradicting the assumption that \(X,X'\) are velcrot.
\end{proof}

The following lemma gives a quintessential example for velcrot subsurfaces:

\begin{lemma}\label{lem: annul homo velcrot}
Let \(X,Z\) be two essential non-sporadic subsurfaces. Suppose in addition that \(Z\) is contained in \(X\) and is homotopic to \(X\). Then \(X,Z\) are velcrot.
\end{lemma}
\begin{proof}
As $Z\subset X$, $\C^{\dagger}(X)\cap \C^{\dagger}(Z)=\C^\dagger(Z)$ is unbounded in $\C^{\dagger}(Z)$. Conversely, for any transverse \(\alpha,\beta\in \C^\dagger(Z)\subset \C^\dagger(X)\), let $P$ be a finite subset of $Z$ such that $\alpha$ and $\beta$ are in minimal position on $X\setminus P$. By \autoref{lem: dagger}, we have
    \[d_{\C^\dagger(Z)}(\alpha,\beta)=d_{\C^s(Z\setminus P)}([\alpha]_{Z\setminus P},[\beta]_{Z\setminus P})=d_{\C^s(X\setminus P)}([\alpha]_{X\setminus P},[\beta]_{X\setminus P})=d_{\C^\dagger(X)}(\alpha,\beta),\]
    where we have used a natural identification between \(\C^s(X\setminus P)\) and \(\C^s(Z\setminus P)\). Since two transverse curves in $\C^{\dagger}(Z)$ can have an arbitrarily large distance, the inclusion \(\C^\dagger(Z)\hookrightarrow\C^\dagger(X)\) is unbounded.
\end{proof}

However, we remark that velcrotness is not an equivalence relation since it is not transitive. This can be easily seen from the following result, which gives a topological characterisation for velcrotness:

\begin{proposition}\label{prop: witness}
Two non-sporadic essential subsurfaces \(X,X'\subset S\) are velcrot if and only if there exists an essential subsurface \(Z\subset X\cap X'\) such that \(Z\) is homotopic to both \(X\) and \(X'\).
\end{proposition}
\begin{proof}
For the ``if'' part, by virtue of \autoref{lem: annul homo velcrot}, $\C^{\dagger}(Z)\subset \C^{\dagger}(X)\cap \C^{\dagger}(X')$ is unbounded in both $\C^{\dagger}(X)$ and $\C^{\dagger}(X')$. Hence, $X$ and $X'$ are velcrot by definition.

For the ``only if'' part, \autoref{lem:velcrot not intersec} implies that the boundaries of \(\partial X\) and \(\partial X'\) only bound bigons or annuli. More specifically, this implies that the boundaries of a small regular neighbourhood of \(\partial X\cup \partial X'\) in \(X\) are peripheral curves in \(X\). These peripheral curves then bound a subsurface \(Z\subset X\) homotopic to \(X\). But since there are curves in \(X\cap X'\) that are essential in both, we can conclude that \(Z\subset X\cap X'\). However, we also know that \(\partial Z\subset X'\) are also peripheral, otherwise they would yield a well-defined projection of \(\partial X\) in \(X'\), contradicting \autoref{lem:velcrot not intersec}. Hence, \(Z\) is also homotopic to \(X'\).
\end{proof}

Since homotopy among subsurfaces is an equivalent relation, we can soon conclude the following results:

\begin{corollary}\label{cor: velcrot implies isotopic}
Two non-sporadic velcrot essential subsurfaces are homotopic. \QEDD
\end{corollary}

\begin{corollary}\label{cor: homo not intersec = velcrot}
Given two homotopic non-sporadic essential subsurfaces, they are velcrot if and only if the boundaries of one surface do not project to an essential arc on the other surface. \QEDD
\end{corollary}
\begin{remark}
We first remark that the homotopy condition cannot be removed, since, for example, two disjoint subsurfaces are never velcrot. In the sequel, we will often use \autoref{cor: homo not intersec = velcrot} implicitly.  Moreover, as a consequence, a small perturbation on the boundary of an essential subsurface within a small regular neighbourhood of it will yield a new surface that is velcrot to the original one. This fact will also be frequently used in the following without mentioning explicitly.
\end{remark}

For any \(\alpha,\beta\in\C^\dagger (S)\) and for any essential non-sporadic subsurface \(X\subset S\), we define the \emph{projection distance} by
\begin{equation}\label{eq: proj dist curve}
d^\pi_X(\alpha,\beta)\coloneqq \diam_{\C^\dagger(X)}\big(\pi_X(\alpha)\cup \pi_X(\beta)\big)>0
\end{equation}
if both  \(\alpha,\beta\) intersect \(X\) essentially; otherwise, we set \(d^\pi_X(\alpha,\beta)\coloneqq\infty\) whenever \(\alpha\) or \(\beta\) do not intersect \(X\) essentially. Moreover, for any triple of essential non-sporadic subsurfaces \(X, Y,Z\subset S\), we define
\begin{equation}\label{eq: proj dist surf}
d^\pi_X(Y,Z)\coloneqq \inf \left\{d^\pi_X(\alpha,\beta): \alpha\subset \partial Y, \beta\subset \partial Z\right\}.
\end{equation}
It is not difficult to see that with the above definition, we then have a triangle inequality
\begin{equation}\label{eq: trian ineq proj dist}
    d^\pi_X(Y,Z)\leq d^\pi_X(Y,Z')+d^\pi_X(Z',Z)
\end{equation}
whenever \(\partial Y,\partial Z,\partial Z'\) all intersect \(X\) essentially. 

\begin{proposition}\label{cor: velcrot same base}
There exists \(M>0\) that verifies the following. Let \(X,X'\) be two velcrot non-sporadic essential subsurfaces. Suppose that \(x,z\in \C^\dagger(S)\) intersect \(X,X'\) essentially, then we have
\[|d^\pi_X(x,z)-d^\pi_{X'}(x,z)|<M.\]
\end{proposition}
\begin{proof}
Note that \(x,z\) also intersect \(Z\) essentially for \(Z\subset X\cap X'\) from \autoref{prop: witness}. Now, let \(\alpha\in \pi_Z(x)\) and \(\beta\in \pi_Z(z)\). A geodesic path \((x_i)_{i}\) in \(\C^\dagger(Z)\) between \(\alpha\) and \(\beta\) will also yield a path in both \(\C^\dagger(X)\) and \(\C^\dagger(X')\), a finite subset \(P\subset Z\) such that \(x_i\)'s are in minimal position on \(Z\), and thus also on \(X,X'\). Hence, by \autoref{lem: dagger}, we have
\[d_{\C^\dagger(Z)}(\alpha,\beta)=d_{\C^s(Z\setminus P)}([\alpha]_{Z\setminus P},[\beta]_{Z\setminus P})=d_{\C^s(X\setminus P)}([\alpha]_{X\setminus P},[\beta]_{X\setminus P})=d_{\C^\dagger(X)}(\alpha,\beta),\]
where \(\C^s(Z\setminus P)\) is the surviving curve graph for \(Z\setminus P\), is defined up to isotopy, and hence is naturally identified with \(\C^s(X\setminus P)\). The same also holds for \(X'\). Moreover, we remark that with the natural inclusion map we have \(\pi_Z(x)\subset \pi_{X}(x), \pi_{X'}(x)\) and  \(\pi_Z(z)\subset \pi_{X}(z), \pi_{X'}(z)\). The desired result follows from the triangle inequality and \autoref{prop: bdd diam}.
\end{proof}

\begin{remark}
We remark that the condition of \(x,z\in \C^\dagger(S)\) intersecting both \(X,X'\) essentially cannot be dropped. Indeed, there are situations where \(X\subset X'\) and \(x\cap X\) contains essential arcs in \(X\) while \(x\cap X'\) does not admit any essential subarc in \(X'\). Conversely, if \(x,z\) intersect \(X'\) essentially, they also intersect \(X\) essentially.
\end{remark}

\subsection{Fine Behrstock inequality}

We offer a combinatorial proof of Behrstock inequality in the fine curve graph setup following the proof of Leininger (see also \cite{mangahas2010uniform}).

Recall that two subsurfaces in \(S\) are {\it overlapping} if their boundaries cannot be made disjoint via homotopy on \(S\). Note that if two subsurfaces \(X,Y\) are overlapping, then \(\partial X\cap Y\) contains an essential arc on \(Y\) and \emph{vice versa}.

\begin{lemma}\label{lem: velcrot small proj}
There exists \(M>0\) such that the following holds. Consider a triple $(X_1,X_2,X_3)$ of non-sporadic essential subsurfaces of $S$ such that $X_i$ and $X_j$ ($i\neq j$) are either overlapping or isotopic. Suppose in addition that $X_2$ and $X_3$ are velcrot and that both $\partial X_2$ and $\partial X_3$ intersect $X_1$ essentially. Then $d^{\pi}_{X_1}(X_2,X_3)<M$.
\end{lemma}
\begin{proof}
Assume first that \(X_1\) and \(X_2,X_3\) are overlapping. Since \(X_2\) and \(X_3\) are velcrot, by \autoref{prop: witness}, we can find a subsurface \(Z\subset X_2\cap X_3\) that is isotopic to both \(X_2\) and \(X_3\). Note that \(X_1\) and \(Z\) are also overlapping but \(\partial Z\) is disjoint from both \(\partial X_2\) and \(\partial X_3\). This indicates that there is an essential arc on \(Y\) that is disjoint from both \(\partial X_2\cap Y\) and \(\partial X_3\cap Y\), forcing $d^{\pi}_{X_1}(X_2,X_3)<M$ for some uniform \(M>0\) after \autoref{prop_disjoint_subsurface_proj}.

If \(X_1\) is isotopic to \(X_2\) and \(X_3\), with the similar arguments, the proof is done when one can find a subsurface \(Z\subset X_2\cap X_3\) isotopic to both \(X_2\) and \(X_3\) but not velcrot to \(X_1\). Assume now that the subsurface \(Z\) constructed in \autoref{prop: witness} is velcrot to \(X_1\), then by \autoref{prop: witness} again, one can find a subsurface \(W\subset X_1\cap Z\subset X_1\cap X_2\cap X_3\) that is isotopic to \(X_i\) for \(i=1,2,3\). This, by \autoref{prop: witness} again, forces \(X_1\) to be velcrot to \(X_2\) and \(X_3\), contradicting our hypothesis on essential intersection. Hence, the subsurface \(Z\) must intersect \(X_1\) essentially, which concludes the proof.
\end{proof}

\begin{remark}
\autoref{lem: velcrot small proj} does not hold if the overlapping condition is removed. One may consider three essential subsurfaces \(X_1, X_2,X_3\) such that \(X_2,X_3\) are velcrot and \(\partial X_2,\partial X_3\) intersect \(X_1\) essentially. In addition, assume that \(X_2\cap X_3\) is disjoint from \(X_1\). Then we may apply a point-pushing homeomorphism on \(X_1\) to modify \(\partial X_3\cap X_1\) so that \(d^\pi_{X_1}(X_2,X_3)\) can be arbitrarily large. See \autoref{fig: point pushing can cause arbitrarily long projection distance}.
\begin{figure}[H]
    \centering
    \includegraphics[width=1.0\textwidth]{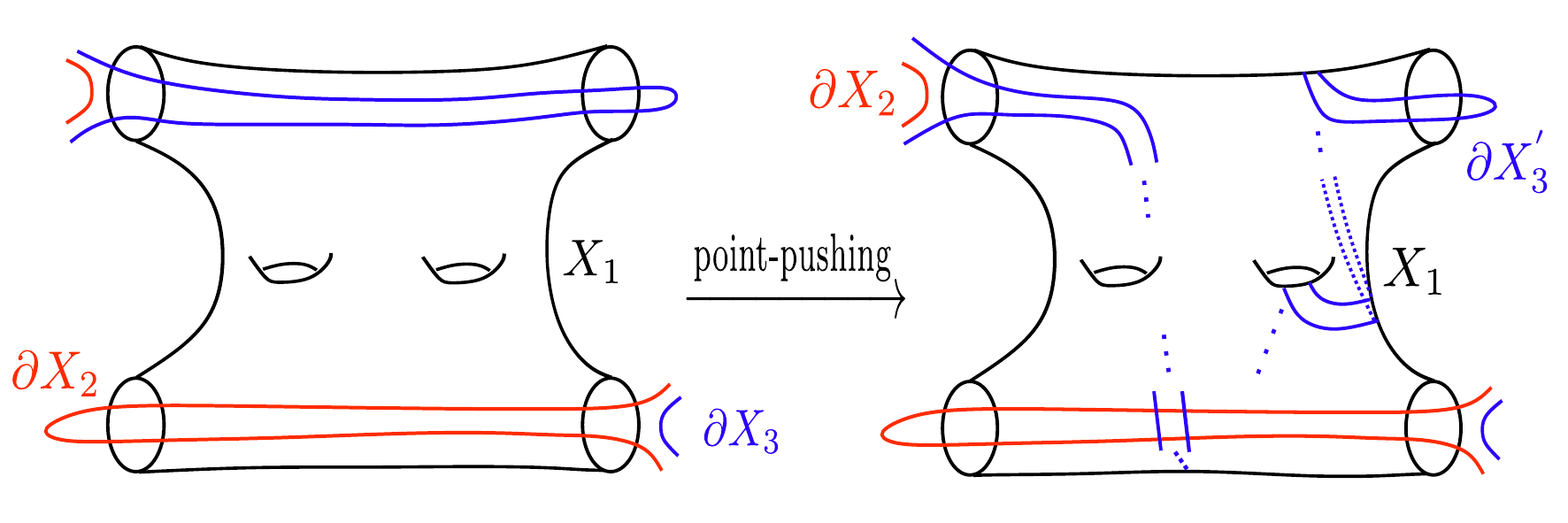}
    \caption{ \(d^\pi_{X_1}(X_2,X_3)\) can be arbitrarily large.}
    \label{fig: point pushing can cause arbitrarily long projection distance}
\end{figure}
\end{remark}

\begin{theorem}[Fine Behrstock's inequality]\label{thm: berhstock}
There exists \(M>0\) such that the following holds. Let \(X_1,X_2,X_3\) be three non-sporadic essential subsurfaces of \(S\) pairwise intersecting each other essentially. Suppose in addition that $X_i$ and $X_j$ ($i\neq j$) are either overlapping or isotopic. If \(d^\pi_{X_1}(X_2,X_3)>M\), then
\[d^\pi_{X_2}(X_1,X_3),d^\pi_{X_3}(X_1,X_2)< M.\]
\end{theorem}
\begin{proof}
By triangle inequality \eqref{eq: trian ineq proj dist} and \autoref{lem: velcrot small proj}, up to a small perturbation, we may assume that \(\partial X_1, \partial X_2, \partial X_3\) are intersecting transversely without triple intersections. If \(\partial X_3\) contains non-essential subarcs on \(X_2\), say \(\alpha\subset \partial X_3\cap X_2\), then \(\alpha\) and a subarc of \(\partial X_2\) will bound a topological disk \(D\subset X_2\). Note that there exists a small regular neighbourhood \(U\) of \(D\) such that \(X_2\setminus U\) is a deformation retract of \(X_2\). So we can isotope \(X_2\) into a new subsurface \(Y\subset X_2\) that is disjoint from \(\alpha\) and the isotopy can be chosen to be supported on \(U\). See \autoref{fig: bigon surg}. 
By this means, we obtain a pair of subsurfaces with strictly one less non-essential intersection. Now by proceeding this operation for \(X_2\) and \(X_3\) for finite steps, we can find two subsurfaces \(Y\subset X_2\) and \(Z\subset X_3\) such that they have no non-essential intersection. By \autoref{prop: witness}, we can conclude that \(Y\) and \(X_2\) are velcrot, so are \(Z\) and \(X_3\). Again, by triangle inequality \eqref{eq: trian ineq proj dist} and \autoref{lem: velcrot small proj}, we may finally assume that \(X_2\) and \(X_3\) do not have non-essential intersection. As \(d^\pi_{X_1}(X_2,X_3)>M\), by \autoref{prop_disjoint_subsurface_proj}, we can take \(M>0\) sufficiently large such that the number of intersections between essential subarcs of \(\partial X_2\cap X_1\) and \(\partial X_3\cap X_1\) is large than \(3\). This in particular implies that we can find an essential subarc of \(\partial X_2\cap X_3\) disjoint from \(\partial X_1\). By \autoref{prop_disjoint_subsurface_proj}, we can conclude that \(d^\pi_{X_3}(X_1,X_2)< M\). By the same means, we also have \(d^\pi_{X_2}(X_1,X_3)< M\).
\end{proof}

\begin{figure}[H]
    \centering
    \includegraphics[scale=0.55]{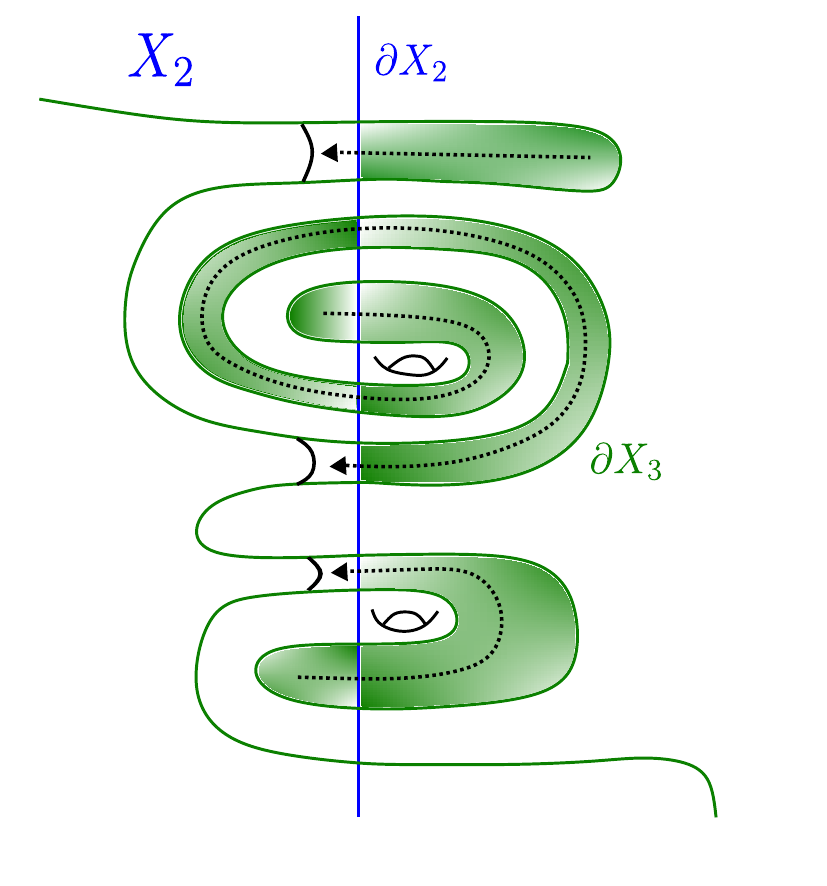}
    \caption{An example of the isotopies operated in \autoref{thm: berhstock}.}
    \label{fig: bigon surg}
\end{figure}

\subsection{Finiteness condition}\label{subsec: finiteness} Unlike the usual setting in \cite{bestvina2015constructing}, without considering isotopy classes, there is generally no finiteness condition for essential subsurfaces. However, with the notion of velcrot subsurfaces, we can show that the for fixed subsurfaces \(Y,Z\), one can choose finitely many subsurfaces such that each subsurface \(X\) with \(d^\pi_X(Y,Z)\gg 1\) is velcrot to one of them. This will serve as a weaker finiteness condition to run the Bestvina--Bromberg--Fujiwara type construction.

\begin{lemma}\label{lem: dual graph}
Let \(\alpha,\beta\in  \C^\dagger(S)\) be two filling transverse curves. Then up to velcrotness, there is a finite canonical family of essential subsurfaces such that their boundaries intersect each subarc of \((\alpha\cup \beta)\setminus(\alpha\cap \beta)\) at most twice and the intersections with \(\alpha\cup \beta\) are essential.
\end{lemma}
\begin{proof}
Let \(\alpha,\beta\in \C^\dagger(S)\) be as above. Then, by definition, the surface \(S\) is cut by \(\alpha\cup \beta\) into topological discs. Now, let \(\mathcal{G}\) be a dual graph of these discs, {\it i.e.} \(\mathcal{G}\) is a graph on \(S\) where each topological disc above contains a unique vertex and two vertices are connected by an edge if the associated topological discs share a common subarc. Now, for each edge of \(\mathcal{G}\), we fix a small regular neighbourhood. Let \(X\) be an essential subsurface as in the claim. On the closure of each topological disc, we can isotope \(X\) to a subsurface \(X'\) such that \(\partial X'\) consists of concatenations of boundaries of the small regular neighbourhoods chosen above and that the intersection pattern between \(\partial X\) and the subarcs \((\alpha\cup \beta)\setminus(\alpha\cap \beta)\) is preserved. Since the possible ways of \(\partial X\) intersecting \(\alpha\cup \beta\) are combinatorially finite, there are only finite possibilities of \(X'\). Moreover, as the isotopies are made within (closed) topological discs, there is no essential intersection of \(\partial X\cap X'\), which implies that \(X,X'\) are velcrot. See \autoref{fig: finiteness}.
\end{proof}

\begin{figure}[H]
    \centering
    \includegraphics[scale=0.28]{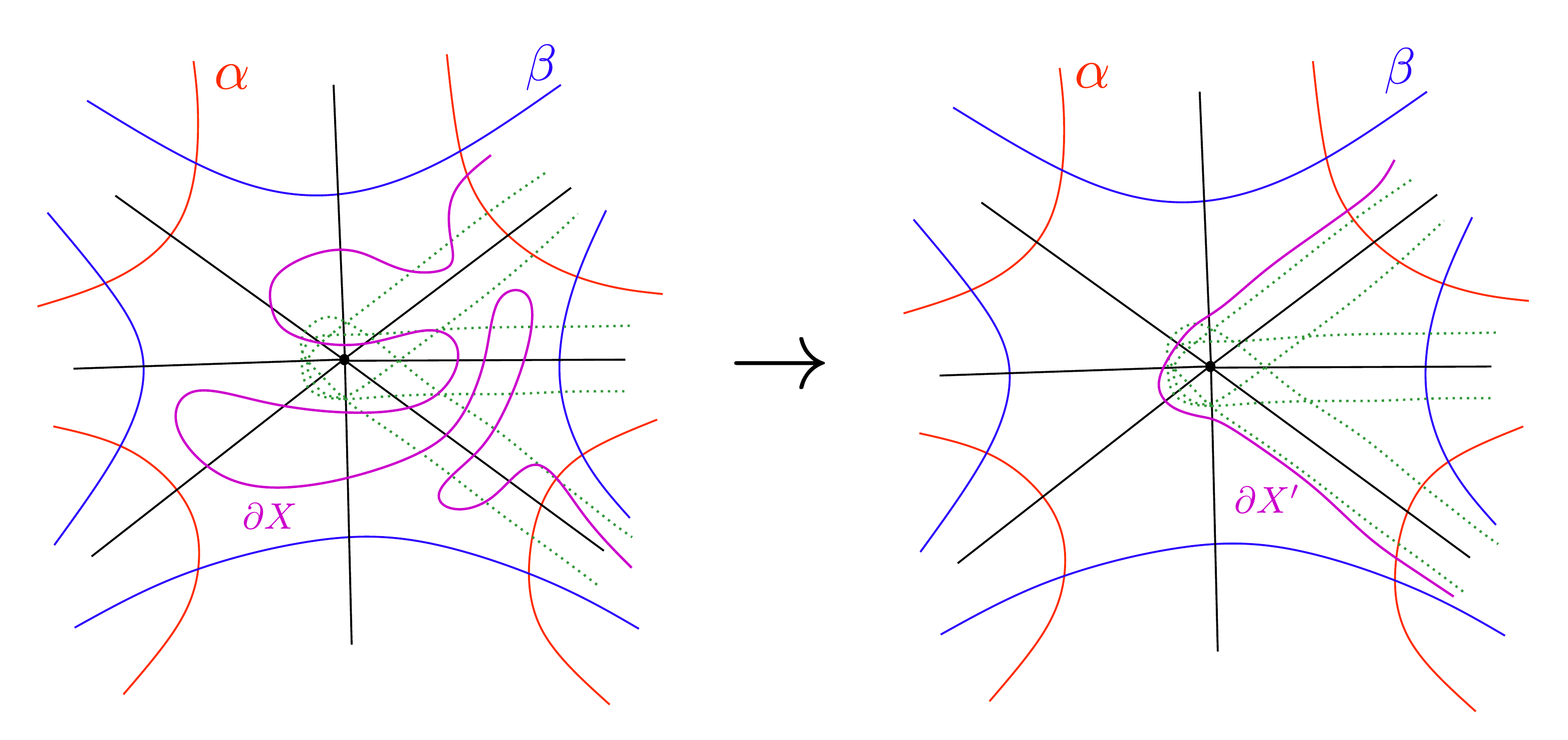}
    \caption{Isotoping subsurface boundaries to boundaries of neighbourhoods of dual graph.}
    \label{fig: finiteness}
\end{figure}

\begin{proposition}\label{prop: finiteness}
There exists a constant \(M>0\), such that the following holds. For any two curves \(\alpha,\beta\in \C^\dagger(S)\), there are finitely many essential non-sporadic subsurface \((X_i)_{i=1}^n\) of \(S\) such that if \(Z\) is an essential non-sporadic subsurface of \(S\) with \(M<d^\pi_Z(\alpha,\beta)<\infty\), then \(Z\) is velcrot to one of \(X_i\)'s. In particular, these \(X_i\)'s can be made pairwise non-velcrot with \(M<d^\pi_{X_i}(\alpha,\beta)<\infty\).
\end{proposition}

\begin{proof}
Note that we can perturb \(\alpha,\beta\) into new curves \(\alpha',\beta'\) intersecting each other transversely such that \(\alpha'\cap \alpha=\emptyset\) and \(\beta'\cap \beta=\emptyset\). By \autoref{prop_disjoint_subsurface_proj}, if \(M<d^\pi_Z(\alpha,\beta)\) for some subsurface \(Z\subset S\), then \(M'<d^\pi_Z(\alpha',\beta')\) for some \(M'> M-24>0\). In turn, if the statement is true for transverse \(\alpha',\beta'\) for any sufficiently large \(M'>0\), then it is also true for \(\alpha,\beta\) and for any sufficiently large \(M>0\). So we may now assume that \(\alpha,\beta\) are transverse.

Moreover, we need only to consider the case when \(\alpha,\beta\) are filling. Indeed, if not, let \(S'\subset S\) be any subsurface containing \(\alpha\cup \beta\) and that \(\alpha,\beta\) fill \(S'\), then whenever \(X\setminus S'\) contains an essential curve on \(X\), we have \(d^\pi_{X}(\alpha,\beta)<M\) by \autoref{thm: bgit}. 

Suppose now that \(\alpha,\beta\) are filling on \(S\), or equivalently, \(\alpha\cup\beta\) cuts \(S\) into finitely many topological discs. Let \(Z\) be any essential non-sporadic subsurface of \(S\) with \(M<d^\pi_Z(\alpha,\beta)<\infty\). We first slightly perturb \(Z\) into \(Z'\subset Z\) such that \(\partial Z'\) intersects \(\alpha,\beta\) transversely without triple intersections. If \(\partial Z'\) intersects any subarc of \((\alpha\cup \beta)\setminus(\alpha\cap \beta)\) more than twice, since we have assumed that \(d^\pi_Z(\alpha,\beta)\) and thus, by \autoref{cor: velcrot same base}, \(d^\pi_{Z'}(\alpha,\beta)\) are sufficiently large, then \(\alpha\) or \(\beta\) admits a subarc bound a bigon with \(\partial Z'\). With the same techniques in \autoref{fig: bigon surg}, we can now find a velcrot subsurface \(Z''\subset Z'\) so that \(Z''\) remains transverse to \(\alpha\cup\beta\) without triple intersections and disjoint from the bigons. Hence, \(\partial Z''\) intersect each subarc of \((\alpha\cup \beta)\setminus(\alpha\cap \beta)\) at most twice.

Now we can isotope \(Z''\) to one of the canonical subsurfaces from \autoref{lem: dual graph}, denoted by \(X\) and there are only finitely many choices of \(X\). We claim that it is also velcrot to \(Z\). First, note that \(Z\) and \(X\) are homotopic. Note that \(Z''\) and \(X\) are velcrot. By \autoref{prop: witness}, we can find \(W\subset Z''\cap X\subset Z\cap X\) that is also homotopic to both \(Z''\) and \(X\), and hence also to \(Z\). By \autoref{prop: witness} again, this implies that \(Z\) and \(X\) are velcrot.
\end{proof}

\section{Fine projection complex}\label{sec-4}

This section will be an adaption of the discussion in \cite[\S 3]{bestvina2015constructing} to our setup. However, we remark that the results we obtained here cannot be implied directly from their theory, and we need to modify various notions introduced in \cite{bestvina2015constructing}. Thus, we provide a relatively detailed proof whenever modifications are involved.

\subsection{Hierarchy}\label{subsec: 4.A}

Let \(\Y\) be a collection of non-sporadic subsurfaces of \(S\) verifying the following two conditions:
\begin{enumerate}[label=(Y\arabic*), topsep=0pt, itemsep=-1ex, partopsep=1ex, parsep=1ex ]
    \item If \(X\in\Y\), then we also have \(f(X)\in \Y\) for every \(f\in\Homeo_0(S)\); \label{Y1}
    \item For each pair \(X,Y\in\Y\), they are either homotopic or overlapping. \label{Y2}
\end{enumerate}
For simplicity, given two subsurfaces \(X,Y\subset S\), we write \(\pi_X(Y)\coloneqq \bigcup_{\alpha\subset \partial Y}\pi_X(\alpha)\).

With these conditions, together with \autoref{cor: homo not intersec = velcrot}, we can easily conclude:
\begin{lemma}
For \(X,Y\in \Y\), either they are velcrot, or \(\partial X\) intersect \(Y\) essentially and \emph{vice versa}. \QEDD
\end{lemma}

Let \(\theta>0\) be a constant larger than \(3M\) for \(M>0\) verifying any results from \autoref{sec: surf surg}.
\begin{definition}\label{def: hierarchy}
For non-velcrot \(X,Z\in \Y\), we define \(\HH(X,Z)\subset \Y\times \Y\) to be the pairs \((X',Z')\) such that \(X',Z'\) are not velcrot and they satisfy one of the following conditions:
\begin{enumerate}[label=(H\arabic*), topsep=0pt, itemsep=-1ex, partopsep=1ex, parsep=1ex ]
    \item\label{H1} \(2 \theta <d^\pi_X(X',Z'),d^\pi_Z(X',Z')<\infty\). 
    \item\label{H2} \(X,X'\) are velcrot and \(2\theta < d^\pi_Z(X,Z')<\infty\). 
    \item\label{H3} \(Z,Z'\) are velcrot and \(2\theta < d^\pi_X(X',Z)<\infty\). 
    \item\label{H4} \(X,X'\) are velcrot and \(Z,Z'\) are velcrot. 
\end{enumerate}
Moreover, we define \(\hh(X,Z)\subset \Y\) to be the subsurfaces \(Y\) contained in a pair from \(\HH(X,Z)\).
\end{definition}

\begin{remark}\label{rem: velcrot in hier}
It is clear from the definition that given two non-velcrot \(X,Z\in \Y\), if \(Y\in \Y\) is velcrot to either \(X\) or \(Z\), then \(Y\in \hh(X,Z)\). Moreover, with our definition of the projection distance above, we can see that \((X',Z')\in \HH(X,Z)\) if and only if \(d_X^\pi(X',Z'), d^\pi_Z(X',Z')>2\theta\).
\end{remark}

We define the \emph{modified projection distance} by
\begin{equation} \label{eq def: modified proj dist}
d_Y(X,Z)\coloneqq \inf \{ d_Y^\pi(X',Z'): (X',Z')\in \HH(X,Z)\},
\end{equation}
if \(Y\notin \hh(X,Z)\) and \(X,Z\) are not velcrot. Otherwise, we set \(d_Y(X,Z)=0\). 

The following proposition shows that upon changing to another pair in the hierarchy, the projection distance does not change much, thus there is a coarse equality between the projection distance and the modified one.

\begin{proposition}\label{prop: mod dist coarse to proj dist}
For any non-velcrot \(X,Z\in \Y\) and any \((X',Z')\in \HH(X,Z)\), if \(Y\in \Y\) is any subsurface that is not velcrot to \(X\) or \(Z\), then
\[d^\pi_Y(X,Z)-d^\pi_Y(X',Z')\leq 2\theta.\]
\end{proposition}
\begin{proof}
If \(Y\) is velcrot to either \(X'\) or \(Z'\), then \(d^\pi_Y(X',Z')=\infty\), there is nothing to proof. So we may assume that \(Y\) is not velcrot to \(X'\) or \(Z'\) as well.

Suppose first that \(Y\in \hh(X,Z)\). Since \(Y\) is not velcrot to either \(X\) or \(Z\), the only possibility is that \ref{H1} is satisfied, \emph{i.e.} there exists \(Y'\in\Y\) such that \(d^\pi_X(Y,Y'),d^\pi_Z(Y,Y')>2\theta\). Then by the fine Behrstock's inequality (\autoref{thm: berhstock}), we can conclude that \(d^\pi_Y(X,Y'),\ d^\pi_Y(Z,Y')\leq \theta\). Now applying \eqref{eq: trian ineq proj dist}, we obtain
\[d^\pi_Y(X,Z)\leq d^\pi_Y(X,Y')+d^\pi_Y(Z,Y')\leq 2\theta\]
As \(d^\pi_Y(X',Z')\geq 0\), the inequality is verified.

The same argument also holds whenever \(d^\pi_Y(X,Z)\leq 2\theta\). So we may now assume that \(d^\pi_Y(X,Z)> 2\theta\). In particular, this implies that \(Y\notin \hh(X,Z)\) and that \(Y\) is velcrot neither to \(X\) nor to \(Z\), after \autoref{rem: velcrot in hier}.

First, let us assume that \(X,X'\) and \(Z,Z'\) are not velcrot. By \eqref{eq: trian ineq proj dist}, as \((X',Z')\in \HH(X,Z)\) verifying \ref{H1} condition, we have
\[d^\pi_X(X',Y)+d^\pi_X(Y,Z')\geq d^\pi_X(X',Z')>2\theta,\]
hence, \(\max\{d^\pi_X(X',Y),d^\pi_X(Y,Z')\}>\theta\). Without loss of generality, we may assume that \(d^\pi_X(X',Y)>\theta\). By the fine Behrstock's inequality, we get \(d^\pi_Y(X,X')<\theta\). Now applying the triangle inequality \eqref{eq: trian ineq proj dist} again, we obtain
\[d^\pi_Y(X,X')+d^\pi_Y(X',Z)\geq d^\pi_Y(X,Z)>2\theta,\]
so \(d^\pi_Y(X',Z)>2\theta-d^\pi_Y(X,X')>\theta\), indicating \(d^\pi_Z(X',Y)<\theta\) after applying the fine Behrstock's inequality. Using the same arguments as the beginning of this paragraph to \(Z\), we also have \(\max\{d^\pi_Z(X',Y),d^\pi_Z(Y,Z')\}>\theta\), but as \(d^\pi_Z(X',Y)<\theta\), this forces \(d^\pi_Z(Z',Y)>\theta\). Now, the fine Behrstock's inequality also applies to get \(d^\pi_Y(Z,Z')<\theta\). Finally, using the triangle inequality for the projection distance \eqref{eq: trian ineq proj dist} again, we have
\[d^\pi_Y(X,X')+d^\pi_Y(X',Z')+d^\pi_Y(Z,Z')\geq d^\pi_Y(X,Z),\]
which in turn implies
\[d^\pi_Y(X,Z)-d^\pi_Y(X',Z')\leq d^\pi_Y(X,X')+d^\pi_Y(Z,Z')<2\theta.\]

Now, let us assume that \(Z,Z'\) are velcrot but \(X,X'\) are not. Then, by our assumption \((X',Z')\in \HH(X,Z)\), this implies that \(d^\pi_X(X',Z')>2\theta\). Applying \eqref{eq: trian ineq proj dist}, we have
\[d^\pi_X(X',Y)+ d^\pi_X(Y,Z')\geq d^\pi_X(X',Z')>2\theta,\]
which implies \(\max\{d^\pi_X(X',Y), d^\pi_X(Y,Z')\}>\theta\). But we also have assumed \(d^\pi_Y(X,Z)>2\theta\), so together with the triangle inequality \eqref{eq: trian ineq proj dist}, \autoref{lem: velcrot small proj} and our assumption \(\theta>3M\), we have
\[d^\pi_Y(X,Z')\geq d^\pi_Y(X,Z)-d^\pi_Y(Z,Z')>2\theta-M>\theta.\]
The fine Behrstock's inequality then implies that \(d^\pi_X(Y,Z')<\theta\), and thus \(d^\pi_X(X',Y)>\theta\). Again, the fine Behrstock's inequality further yields \(d^\pi_Y(X,X')<\theta\). We once again apply the triangle inequality \eqref{eq: trian ineq proj dist} to see that
\[d^\pi_Y(X,X')+d^\pi_Y(X',Z)\geq d^\pi_Y(X,Z),\]
and therefore
\begin{align*}
d^\pi_Y(X,Z)-d^\pi_Y(X',Z')& \leq d^\pi_Y(X,Z)-d^\pi_Y(X',Z)-d^\pi_Y(Z',Z)\\
&<\theta-M<2\theta.
\end{align*}
By symmetry, the situation when \(X,X'\) are velcrot but \(Z,Z'\) are not can be discussed in the same way.

Finally, if \(X,X'\) and \(Z,Z'\) are velcrot, the desired result is a simple application of \autoref{lem: velcrot small proj} after the assumption that \(3M<\theta\). 
\end{proof}

Let us define \(x\pt y\) or \(y\st x\) if \(x-y\) is bounded above by a constant that only depends on \(\theta>0\). We also define \(x\simt y\) if both \(x\pt y\) and \(x\st y\).

Next, for a large \(K>0\) and any \(X,Z\), we define \(\Y_K(X,Z)\) to be the set of \(Y\in \Y\) such that \(d_Y(X,Z)>K\). In particular, we can see that if \(X,Z\) are velcrot, then \(\Y_K(X,Z)=\emptyset\); conversely, if \(Y\in\Y_K(X,Z)\neq \emptyset\), then \(X,Y,Z\) are pairwise not velcrot.

\begin{theorem}\label{thm: BBF prerequi}
There exists a \(\Theta > 0\), depending only on \(\theta\), such that the
following properties hold:
\begin{enumerate}[label=(\Roman*), topsep=0pt, itemsep=-1ex, partopsep=1ex, parsep=1ex ]
\item\label{I} {\bf Symmetry}.
\[d_Y(X,Z)=d_Y(Z,X)\]
\item\label{II} {\bf Coarse equality}. For any pairwise non-velcrot triple \(X,Y,Z\)
\[d_Y^\pi(X,Z)\pt d_Y(X,Z)\leq d^\pi_Y(X,Z).\]
If \(X,Z\) are velcrot but \(Y\) is velcrot to neither \(X\) nor \(Z\), then
\[d_Y^\pi(X,Z)\simt d_Y(X,Z)=0.\]
\item\label{III} {\bf Velcrot coarse equality}. For any velcrot \(Y,Y'\in \Y\) and any pair \(X,Z\in \Y\) not velcrot to \(Y,Y'\),
\[d_Y(X,Z)\simt d_{Y'}(X,Z).\]
\item\label{IV} {\bf Coarse triangle inequality}. For \(Y\) not velcrot to \(Z\),
\[d_Y(X,Z)+d_Y(Z,W)\st d_Y(X,W).\]
\item\label{V} {\bf Inequality on triples}.
\[\min\{d_Y(X,Z),d_Y(Z,W)\}\simt 0.\]
\item \label{VI} {\bf Finite velcrot covering}. For any \(X,Z\in \Y\), there exist finitely many pariwise non-velcrot \(Y_i\in \Y_\Theta(X,Z)\) such that every \(Z\in \Y_\Theta(X,Z)\) is velcrot to one of \(X_i\)'s.
\item \label{VII} {\bf Monotonicity}. If \(Y\in \Y_\Theta(X,Z)\neq \emptyset\), then for any \(W\in \Y\) not velcrot to \(X\) or \(Z\), both \(d_W(X,Y), d_W(Y,Z)\leq d_W(X,Z)\).
\item \label{VIII} {\bf Order}. Given \(X,Z\in \Y\) non-velcrot, for any pairwise non-velcrot \(Y_i\)'s in \(\Y_\Theta(X,Z)\), there is a total order on \((Y_i)\) such that if \(Y_0<Y_1<Y_2\), then
\[d_{Y_1}(X,Z)\pt d_{Y_1}(Y_0,Y_2)\leq d_{Y_1}(X,Z)\]
and
\[d_{Y_0}(Y_1,Y_2), d_{Y_2}(Y_0,Y_1)\simt 0.\]
Moreover, the order can be extended to \(X\) and \(Z\) with \(X\) being the minimal element and \(Z\) the maximal.
\item \label{IX} {\bf Barrier property}. If \(Y\in \Y_\Theta(X_0,Z)\) and \(Y\in \Y_\Theta(X_1,Z)\), then \(d_Z(X_0,X_1)<\Theta\).
\end{enumerate}
\end{theorem}
\begin{proof}
\ref{I} is deduced directly from the definition. The right-hand side of the first part of \ref{II} is clear from the definition, while the left-hand side is a consequence of \autoref{prop: mod dist coarse to proj dist}. The second part of \ref{II} is a reformulation of \autoref{lem: velcrot small proj} with our new notation. \ref{III} is a consequence of \ref{II} and \autoref{cor: velcrot same base}. \ref{IV} is a result of the triangle inequality \eqref{eq: trian ineq proj dist} and the coarse equality \ref{II}. \ref{V} can also be deduced directly from the definition of the modified projection distance and the fine Behrstock's inequality. \ref{VI} is also a consequence of \autoref{prop: finiteness}.

For \ref{VII}, if \(W\) is velcrot to \(Y\), then the proof is trivial. Suppose now \(W,Y\) are not velcrot. We claim that \(\Theta>4\theta\) verifies the condition that if \(Y\in\Y_\Theta(X,Z)\neq \emptyset\), then
\[\HH(X,Z)\subset \HH(X,Y)\cap \HH(Y,Z).\]
Indeed, the distance \(d_W(X,Z)\) is the infimum of \(d^\pi_W(X',Z')\) for \((X',Z')\in\HH(X,Z)\), while \(d_W(X,Y)\) and \(d_W(Y,Z)\) are infimum taken over \(\HH(X,Y)\) and \(\HH(Y,Z)\) respectively, and the inclusion will give us the desired inequality. Now, let us verify it. If \((X', Z') \in \HH(X, Z)\), then by \autoref{prop: mod dist coarse to proj dist}, we have
\[d_Y^\pi(X,Z)-d^\pi_Y(X',Z')<2\theta.\]
But \(d^\pi_Y(X,Z)\geq d_Y(X,Z)>\Theta>4\theta\), then \(d^\pi_Y(X',Z')>2\theta\). In any case (despite the velcrotness between the surfaces), this indicates \((X',Z')\in \HH(X,Y)\cap \HH(Y,Z)\).

For \ref{VIII}, since \((Y_i)\)'s are already pairwise non-velcrot and none of them is velcrot to either \(X\) or \(Z\), the same arguments in the proof of \cite[Theorem 3.3]{bestvina2015constructing} also apply here.

For \ref{IX}, if \(X_0,X_1\) are velcrot, then the desired result is nothing more than the coarse equality \ref{II}; otherwise, they are pairwise non-velcrot, and the discussion goes \emph{verbatim} as in \cite[Theorem 3.3]{bestvina2015constructing}.

Finally, we take \(\Theta>0\) such that it bounds \emph{three times} of all the differences appear in the coarse relation in \ref{I}--\ref{IX}.
\end{proof}

\begin{remark}\label{rem: order}
More precisely, the order in \ref{VIII} is defined as below: fix \(\theta'>4\theta>0\), for two non-velcrot \(W,Y\in \Y_\Theta(X,Z)\), we say that \(W<Y\) if one, and hence all, of the following equivalent conditions is satisfied:
\begin{enumerate}[label=(\roman*), topsep=0pt, itemsep=-1ex, partopsep=1ex, parsep=1ex ]
\item \(d_W(X,Y)>\theta'\,\);
\item \(d_Y(X,W)\leq\theta'\,\);
\item \(d_Y(W,Z)>\theta'\,\);
\item \(d_W(Y,Z)\leq \theta'\,\).
\end{enumerate}
Moreover, we say that \(W\leq Y\) in \( \Y_\Theta(X,Z)\) if either \(W<Y\) or else \(W,Y\) are velcrot.
\end{remark}

\subsection{Quasi-tree}\label{subsec: 4.B} Let \(\Theta\) still be the constant in \autoref{thm: BBF prerequi}. For \(K\geq 2\Theta\), let us define the following graph:
\begin{definition}
    The \emph{fine projection complex} \(\PY\) is a graph where the vertices are elements in \(\Y\) and an edge is connected between \(X,Z\in\Y\) if \(\Y_K(X,Z)=\emptyset\). Denote the simplicial distance on \(\PY\) by \(\dP\).
\end{definition}

In the following, we will discuss the metric geometric properties of \(\PY\).

For a subset \(U\subset \Y\), we define \(|U|_\vel\) to be the minimal \(k\geq 0\) needed so that there exists \((Y_i)_{i=1}^k\subset U\) such that any \(Y\in U\) is velcrot to some \(Y_i\), with the convention \(|\emptyset|_\vel=0\). By \ref{VI}, we see that \(|\Y_K(X,Z)|_\vel<\infty\) for any \(X,Z\in \Y\).

\begin{proposition}\label{prop: PK connect}
For any \(X,Z\in \Y\), we have \(\dP(X,Z)\leq |\Y_K(X,Z)|_\vel+1\). In particular, the graph \(\PY\) is connected.
\end{proposition}
\begin{proof}
If \(\Y_K(X,Z)=\emptyset\), then they are connected in \(\PY\) and 
\[\dP(X,Z)= |\Y_K(X,Z)|_\vel+1=1.\]
Suppose now that \(\Y_K(X,Z)\neq\emptyset\). Then there exists \((Y_i)_{i=1}^k\subset \Y_K(X,Z)\) pairwise non-velcrot with \(k=|\Y_K(X,Z)|_\vel\) such that any \(Y\in \Y_K(X,Z)\) is velcrot to one of \(Y_i\)'s. Moreover, as \(d_{Y_i}(X,Z)>0\), \(Y_i\)'s are not velcrot to \(X\) or \(Z\). By \ref{VIII}, we may assume that \(Y_i<Y_j\) whenever \(i<j\).

We claim that \(X, Y_1,\dots, Y_k,Z\) is a path connecting \(X\) to \(Z\) in \(\PY\). We first show that \(\Y_K(X,Y_1)=\emptyset\). Suppose for contradiction that \(Y\in \Y_K(X,Y_1)\neq\emptyset\). Then \(Y\) is not velcrot to \(Y_1\). By monotonicity \ref{VII}, we have \(d_Y(X,Z)\geq d_Y(X,Y_i)>K\), which implies that \(Y\in \Y_K(X,Z)\) and is velcrot to some \(Y_i\). However, by \ref{III}, \(d_Y(X,Y_1)\simt d_{Y_i}(X,Y_1) > K>\theta'\), where \(\theta'\) is as in \autoref{rem: order}, showing that \(Y_i<Y_1\), a contradiction. For the adjacency between \(Y_i\) and \(Y_{i+1}\), as well as between \(Y_k\) and \(Z\), the discussion remains very similar, if it is not exactly the same.
\end{proof}

For two vertices in \(\PY\), \autoref{prop: PK connect} gives an upper bound for the distance between them, but does not produce a lower bound estimate. The following notion of \emph{guard} is introduced to give a coarse positioning for geodesic paths in \(\PY\). Roughly speaking, a guard \(W\) for \(Y\) is a vertex such that its projection is always close to the projection of \(Y\) from every viewpoint.

\begin{definition}
We say that \(W \in \Y\) is a \emph{guard} for \(Y\) if for every vertex \(X \in \Y\) with \(W \in \Y_\Theta(X, Y )\) and every \(Z \in \Y_K (X, Y ) \subset \Y_\Theta(X, Y )\), one has \(Z \leq W\) in \(\Y_\Theta(X, Y )\).    
\end{definition}

The following lemma offers a sufficient condition for a vertex \(W\) to be a guard for another vertex \(Y\):

\begin{lemma}\label{lem: K/2 empty means guard}
For \(K>\) sufficiently large and vertices \(X, Y, Z\) and \(W\) in \(\Y\), if \(W \in \Y_\Theta(X, Y )\), \(Z \in \Y_K (X, Y )\) and \(W < Z\) in \(\Y_\Theta(X, Y )\), then \(Z \in \Y_{K/2}(W, Y )\).

In particular, if \(\Y_{K/2}(W, Y )=\emptyset\) then \(W\) is a guard for \(Y\).
\end{lemma}
\begin{proof}
    With the assumption above, we can deduce that \(Z,W\) and \(X,Y\) are not velcrot. Then, appealing to \ref{VIII} with \(W<Z<Y\) with respect to the order in \(\Y_\Theta(X, Y )\cup\{X,Y\}\), we have
    \[d_Z(W,Y)\st d_Z(X,Y)>K,\]
    which, by making \(K\) sufficiently large with respect to \(\theta\), implies \(d_Z(W,Y)>K/2\).
\end{proof}

The following lemma gives an example of guard:
\begin{lemma}\label{lem: maximal is guard}
The following holds for sufficiently large \(K>0\). Let \(X,Z\in \Y\) be non-velcrot. Suppose that \((Y_i)\subset \Y_{K/2}(X,Z)\) are pairwise non-velcrot and that each \(Y\in \Y_{K/2}(X,Z)\) is velcrot to one of \(Y_i\)'s. Then the minimal element in \((Y_i)\) is a guard for \(X\), while the maximal element is a guard for \(Z\). 
\end{lemma}
\begin{proof}
Due to symmetry, we will only demonstrate the reason why the maximal element in \((Y_i)\) is a guard for \(Z\). Assume that \(Y_n\) is the maximal element. Appealing to \autoref{lem: K/2 empty means guard}, it suffices to show that \(\Y_{K/2}(Y_n,Z)=\emptyset\). Suppose for contradiction that \(W\in\Y_{K/2}(Y_n,Z)\). If \(W\) is not velcrot to \(X\), then for \(K>0\) sufficiently large, since \(Y_n\in \Y_{K/2}(X,Z)\subset \Y_\Theta(X,Z)\), the monotonicity \ref{VII} indicates that \(K/2\leq d_W(Y_n,Z)\leq d_W(X,Z)\), forcing \(W\in \Y_{K/2}(Y_n,Z)\) and \(W\leq Y_n\), which by \autoref{rem: order} means that \(d_W(Y_n,Z)\leq \theta'<K/2\), contradiction. If \(W\) and \(X\) are velcrot, then by the inequality on triples \ref{V}, \(d_W(Y_n,Z)>K/2\) will imply that \(d_{Y_n}(W,Z)\simt d_{Y_n}(X,Z)\simt 0\), contradicting \(Y_n\in \Y_{K/2}(X,Z)\).
\end{proof}
\begin{remark}\label{rem: guard adjacent}
The similar argument will yield that if \(W\) is a guard for \(Y\), then \(\Y_{K}(W,Y)=\emptyset\) and thus \(\dP(W,Y)=1\).
\end{remark}

\begin{lemma}\label{lem: PC dist control}
Let \(X_0,X_1\in \Y\) be adjacent in \(\PY\). Suppose that \(W\in \Y\) is such that \(\dP(X_0,W)\geq 3\). Then
\[d_W(X_0,X_1)\simt 0\]
and 
\[d_W(X_0,Z)\simt d_W(X_1,Z)\]
for all \(Z\in \Y\).
\end{lemma}
\begin{proof}
First, it is clear that \(W\) is not velcrot to either \(X_0\) or \(X_1\).

If \(X_0\) and \(X_1\) are velcrot, then \(d_W(X_0,X_1)=0\) and  \(d_W(X_0,Z)\simt d_W(X_1,Z)\) is implied by \autoref{lem: velcrot small proj}.

Assume now that \(X_0,X_1\) are not velcrot. Since \(\dP(X_0, W ) \geq 2\), there exists \(Y \in \Y_K (X_0, W )\). If \(Y,X_1\) are velcrot, then by \ref{III}, we have
\[d_{X_1}(X_0,W)\simt d_Y(X_0,W)>K,\]
which, for sufficiently large \(K>\Theta\), implies \(d_{X_1}(X_0,W)>\Theta\). Hence, by the inequality on triples \ref{V}, this indicates \(d_W(X_0,X_1)<\Theta\). Otherwise, if \(d_W (X_0, X_1) > \Theta\) and \(Y,X_1\) are not velcrot, then by monotonicity \ref{VII} we have
\[d_Y (X_0, X_1) \geq d_Y (X_0, W ) > K,\]
which contradicts \(\dP(X_0, X_1) = 1\). Therefore, we have \(d_W (X_0, X_1) < \Theta\). Applying the coarse triangle inequality \ref{IV} we have
\[d_W (Z, X_0) + d_W (X_0, X_1) \st d_W (Z, X_1),\]
which implies half of the second inequality. The other half is proved by
swapping \(X_0\) and \(X_1\).
\end{proof}

\begin{remark}\label{rem: non-velcrot PC dist control}
If \(X_0,X_1,W\) in \autoref{lem: PC dist control} are given \emph{a priori} pairwise non-velcrot, then the result of \autoref{lem: PC dist control} also holds even given that \(\dP(X_0,W)\geq 2\). In fact, this is the classical case in \cite[Lemma~3.10]{bestvina2015constructing}.
\end{remark}

\begin{lemma}\label{lem: barrier induction}
If \(K>0\) is sufficiently large, the following holds. Let \(Z\) be in \(\Y\) and \(X_0, X_1\) be adjacent vertices in \(\PY\) with \(\dP(X_i, Z) \geq 4\). Let \(W\) be a guard for \(Z\) such that \(W \in \Y_{K/2}(X_0, Z)\). If \(W \notin \Y_{K/2}(X_1, Z)\), then there exists a guard \(W'\) for \(Z\) such that \(W' \in \Y_{K/2}(X_1, Z)\) and \(W \in \Y_\Theta(W', Z)\).
\end{lemma}
\begin{proof}
Since \(\dP(W,Z)=1\) and \(\dP(X_0,Z)\geq 4\), we have \(\dP(X_0,W)\geq 3\), so we can apply \autoref{lem: PC dist control}, which yields
\[d_W(X_1,Z)\simt d_W(X_0,Z)>K/2.\]
This, in turn, indicates that \(W\in \Y_\Theta(X_1,Z)\) for sufficiently large \(K>0\).

Similarly, we also have \(\dP(X_1,W)\geq 3\). We see that there exists \(W'\in \Y_{K/2}(X_1,Z)\subset \Y_\Theta(X_1,Z)\) not velcrot to \(W\) such that \(W'<W\), otherwise \autoref{rem: guard adjacent} and \autoref{rem: guard adjacent} imply that \(\dP(W,X_1)=1\), a contradiction. Due to the finiteness condition \ref{VI}, we may assume that there is no \(V\in \Y_{\Theta}(X_1,Z)\) such that \(W'<V<W\) in \(\Y_{\Theta}(X_1,Z)\). By the order property \ref{VIII}, we conclude
\[d_W(W',Z)\st d_W(X_1,Z)\st K/2,\]
which indicates \(W\in \Y_\Theta(W',Z)\) when \(K>0\) is sufficiently large.

Now, we claim that \(W'\) is a guard for \(Z\). Since \(W'>W\) in \(\Y_\Theta(X_1,Z)\), according to \autoref{rem: order}, for any \(X\) velcrot to \(W\), \(d_{W'}(X,Z)\simt 0\). So it suffices to verify the maximality condition of \(W'\) for any \(X\) not velcrot to \(W\) with \(d_{W'} (X, Z) > \Theta\). By monotonicity \ref{VII}, we also have \(d_W (X, Z) > d_W(W',Z)>\Theta\). Let \(V\in \Y_K(X,Z)\). Since \(W\) is a guard for \(Z\), \(V\leq W\) in \(\Y_\Theta(X,Z)\). Note that \(V\) cannot be velcrot to \(W\) as \(d_{VI}(X,Z)>K\). Assume for contradiction that \(V>W'\) in \(\Y_{\Theta}(X,Z)\), then by \autoref{lem: K/2 empty means guard}, \(V\in \Y_{K/2}(W',Z)\subset \Y_{K/2}(X_1,Z)\), where the inclusion is a consequence of monotonicity \ref{VII} as before. But this, as described in \autoref{rem: order}, also implies \(W'<V\leq W\) in \(\Y_\Theta(X_1,Z)\), contradicting to our choice of \(W'\). So \(V<W'\)  in \(\Y_{\Theta}(X,Z)\). This verifies the maximality condition and implies that \(W'\) is a guard for \(Z\).
\end{proof}

\begin{definition}
A \emph{barrier} between a path \(\{X_0, \dots , X_k\}\) and a vertex \(Z\) is a vertex \(Y\) such that \(Y \in \Y_\Theta(X_i, Z)\) for all \(i = 0, \dots, k\).
\end{definition}

\begin{lemma}\label{lem: barrier implies small proj}
If there is a barrier between a path \(\{X_0, \dots , X_k\}\) and a vertex \(Z\), then \(d_Z(X_i,X_j)\leq \Theta\) for all \(i,j\).
\end{lemma}
\begin{proof}
Since \(Y \in \Y_\Theta(X_i, Z)\) for all \(i\), \(Y\) is not velcrot to \(Z\) or \(X_i\), and we have \(d_Z(Y,X_i)\leq d_Z^\pi(Y,X_i)<\theta\) by \autoref{thm: berhstock} and \ref{II} in \autoref{thm: BBF prerequi}. Now, the desired result follows from the coarse triangle inequality \ref{IV}.
\end{proof}

\begin{proposition}\label{prop: barrier}
The following holds if \(K>0\) is sufficiently large. Let \(\{X_0, X_1, \dots , X_k\}\) be a path in \(\PY\) and \(Z\) a vertex of \(\PY\) such that \(\dP(Z, X_i ) \geq 4\) for all \(i\). Then there is a barrier \(W\) between the path and \(Z\). In particular, \(d_Z (X_0, X_i) \simt 0\) for all \(i\).
\end{proposition}
\begin{proof}
We will inductively choose a family of guards \(W_i\) for \(Z\) such that for each \(1\leq i\leq k\), \(W_i \in \Y_{K/2}(X_i, Z)\), and if \(i > j\) then either \(W_i = W_j\) or \(W_j \in \Y_\Theta(W_i, Z)\).

First, let \(W_0\in \Y_{K/2}(X_0, Z)\) be a guard for \(Z\) as in \autoref{lem: maximal is guard}. Now, for the induction step, suppose that \(W_0,\dots, W_i\) for some \(i<k\) have been chosen.If \(W_{i} \in \Y_{K/2}(X_{i+1}, Z)\), then we simply set \(W_{i+1}\coloneqq W_i\). Otherwise, by \autoref{lem: barrier induction}, there
exists a guard \(W_{i+1}\in \Y_{K/2}(X_{i+1},Z)\) for \(Z\) with \(W_i\in \Y_{\Theta}(W_{i+1},Z)\). We also remark that \(W_{i+1}\) is not velcrot to any \(W_j\) for \(j< i\), otherwise by \autoref{thm: BBF prerequi}, we have
\[\theta>d_{W_{i+1}}^\pi(W_i,Z)\geq d_{W_{i+1}}(W_i,Z)\simt d_{W_j}(W_i,Z).\]
Nevertheless, by our induction hypotheses \(d_{W_j}(W_i,Z)>\Theta\) for \(j<i\) and by our assumption on \(\Theta>0\), this will yield a contradiction. Hence, \(W_0,\dots, W_{i+1}\) are pairwise non-velcrot, so by applying the monotonicity condition \ref{VII}, we can also conclude that \(W_j \in \Y_\Theta(W_{i+1}, Z)\) for all \(j<i\).

Let \(W\coloneqq W_0\). Using the monotonicity condition \ref{VII} again, we can deduce that for any \(\Y_\Theta(W_i,Z)\subset \Y_\Theta(X_i,Z)\). Therefore, we conclude that \(W\in \Y_\Theta(X_i,Z)\) for all \(i\) and thus \(W\) is a barrier between \(\{X_0, X_1, \dots , X_k\}\) and \(Z\).
\end{proof}

We recall that a \emph{quasi-tree} is a geodesic metric space that is quasi-isometric to a tree. One useful characterisation of a quasi-tree is the \emph{bottleneck criterion} introduced by Manning in \cite{manning2005geometry}. To be more precise, a geodesic metric space \(X\) is said to satisfy the bottleneck criterion if there exists a constant \(\Delta\geq0\) such that for any two points \(x,y\in X\), there exists a midpoint \(z\in X\) between \(x\) and \(y\) such that any paths in \(X\) connecting \(x\) to \(y\) intersect the \(\Delta\)-neighbourhood of \(z\). With this characterisation, we can show the following:

\begin{theorem}
For sufficiently large \(K>0\), the graph \(\PY\) is a quasi-tree.
\end{theorem}
\begin{proof}
Let \(X,Z\) be any two vertices on \(\PY\). By \autoref{prop: PK connect}, we can find a path \(\gamma\subset \PY\) connecting \(X\) to \(Z\). Suppose that \(X=X_0,X_1,\dots, X_n=Z\) is a path from \(X\) to \(Z\). This path has to stay inside of \(3\)-neighbourhood of \(\gamma\). Otherwise, there is \(Y\in \gamma\subset\Y_K(X,Z)\) such that \(\dP(X_i,Y)\geq 4\) for all \(X_i\), which by \autoref{prop: barrier} indicates that \(d_Y(X,Z)<\Theta\), a contradiction.
\end{proof}

\subsection{Puncturing}\label{subsec: puncturing}

In this subsection, we will mainly treat the unboundedness of the quasi-tree that we have constructed above. But the arguments from \cite{bestvina2015constructing} cannot be adapted directly. In \cite{bestvina2015constructing}, for any pair \(X,Z\in \Y\), if \(Y\in \Y_{K'}(X,Z)\) for some sufficiently large \(K'>K\), then they showed that \(Y\) must lie on \emph{any} geodesic path between \(X\) and \(Z\) in the quasi-tree. However, this cannot be true for the fine projection complex. Indeed, by perturbing \(Y\) a little, we can always construct another geodesic path that does not pass through the original \(Y\).

More precisely, we only have:

\begin{lemma}\label{lem: huge YK velcrot in geodesic}
There exists \(K'>0\) such that if \(Y\in \Y_{K'}(X,Z)\), then every
geodesic from \(X\) to \(Z\) in \(\PY\) contains an element velcrot to \(Y\).
\end{lemma}
\begin{proof}
Suppose that \(X=X_0,X_1,\dots, X_n=Z\) is a geodesic path in \(\PY\) and that it does not contain element velcrot to \(Y\). We claim that \(d_Y(X,Z)\pt 8K\).

If \(\dP(X_i,Y)\geq 4\) for all \(i\), then \autoref{prop: barrier} indicates that \(d_Y(X,Z)\simt 0 \pt 8K\). So now, we may assume that \(\dP(X_i,Y)<4 \) for some \(i\). Let \(i_0\) be the first \(i\) such that \(\dP(X_i,Y)<4\), and \(i_1\) be the last one. Suppose now \(i_0>0\) and \(i_1<n\). Note that \(\dP(X_{i_0},X_{i_1})\leq 6\). Now, by \autoref{prop: barrier}, we again have \(d_Y(X,X_{i_0 -1})\simt 0\) and  \(d_Y(X_{i_1 +1},Z)\simt 0\).

Now, for \(X_j\) such that \(i_0-1\leq j\leq i_1\), as \(\Y_K(X_j,X_{j+1})=\emptyset\), we have \(d_Y(X_j,X_{j+1})<K\). Since no \(X_j\) is velcrot to \(Y\), we can apply the coarse triangle inequality \ref{IV} to get
\[d_Y(X,Z)\pt d_Y(X,X_{i_0-1})+\sum_{j={i_0-1}}^{i_1} d_Y(X_j,X_{j+1})+d_Y(X_{i_1+1},Z)\leq 8K.\]
Hence, by taking \(K'\st 8K\), we can conclude the desired result.

For the situations where \(i_0=0\) or \(i_1=n\), the proof can be easily adapted and we leave it to readers.
\end{proof}

To deduce that the quasi-tree we constructed is not bounded, we will use the following puncturing technique, which is found to have its roots in the original paper \cite{bowden2022quasi}. 

Suppose that we have a finite collection of points \(F\) on the surface \(S\). Let \(\widetilde{\Y}\) be a subcollection of \(\Y\) such that \(X\neq Z\in \widetilde{\Y}\) are not velcrot, and the boundaries \(\partial X,\partial Z\) are disjoint from \(F\) and are in minimal positions on \(S\setminus F\). In particular, let \(\widetilde{\Y}\) be a subcollection of the geodesic representatives of \(\Y\) with respect to some hyperbolic metric on \(S\setminus F\). 

Consider the \emph{surviving curve graph} of these subsurfaces. It is a graph \(\C^s(X\setminus F)\) where the vertices are isotopy classes of essential simple closed curves on \(X\setminus F\) that also remains essential in \(X\), and we attach an edge between two vertices when they have disjoint representatives. Since any distinct \(X,Z\in \widetilde{\Y}\) are not velcrot, this implies that \(\partial X\) intersect \(Z\) essentially, and \emph{vice versa}, or equivalently, they have well-defined subsurface projection on their surviving curve graphs. 

Starting from a collection \(\widetilde{\Y}\) satisfying several axioms, a renowned construction by Bestvina, Bromberg and Fujiwara will produce a new geodesic metric space, called the {\it projection complex of \(\widetilde{\Y}\)}.

\begin{theorem}[\cite{bestvina2015constructing}]\label{thm: BBF}
Let $\widetilde{\Y}$ be a collection of geodesic metric spaces together with its projection maps between different spaces. Suppose that there is some $M\geq 0$ such that the following three axioms hold:
\begin{enumerate}[label=(P\arabic*), topsep=0pt, itemsep=-1ex, partopsep=1ex, parsep=1ex ]
    \setcounter{enumi}{-1}
    \item For any distinct pair $X, Y\in \widetilde{\Y}$, we have $\diam(\pi_{Y}(X))\leq M$.\label{P0}
    \item For any distinct triple \(X,Y,Z\in\widetilde{\Y}\), if $d^\pi_{Y}(X,Z)>M$, then $d^\pi_{X}(Y,Z)\leq M$.\label{P1}
    \item For any distinct $X,Z\in \mathcal{Y}$, the set $\{Y\in \mathcal{Y}: d^\pi_{Y}(X,Z)>M\}$ is finite.\label{P2}
\end{enumerate}
Then for sufficiently large \(R>0\), there is a quasi-tree \(\mathcal{P}_R(\widetilde{\Y})\), where vertices are elements in \(\widetilde{\Y}\) and an edge is attached to two vertices \(X,Z\in \widetilde{\Y}\) if
\[\widetilde{\Y}_R(X,Z)\coloneqq\{Y\in \widetilde{\Y} : \widetilde{d}_Y(X,Z)>R\}=\emptyset,\]
where \(\widetilde{d}_Y(X,Z)\) is the modified projection distance similar to what is defined above.
\end{theorem}

Following the discussion in \cite{long2025connected}, we can see that the subsurface projection to the surviving curve graphs satisfies the two axioms \ref{P0}, while \ref{P1} and \ref{P2} are classical from Leininger's arguments (see for example \cite[Lemma 5.2 \& Lemma 5.3]{bestvina2015constructing}). So applying \autoref{thm: BBF}, we can build a quasi-tree \(\mathcal{P}_R(\widetilde{\Y})\).

To be more precise, here we can assume that \(R>K+\theta\) for \(K>0\) from \autoref{thm: BBF} and \(M>0\) is from \autoref{sec: surf surg}. Let us similarly define
\[\widetilde{d}^\pi_X(X',Z')\coloneqq\diam_{\C^s(X\setminus F)}(\bigcup\{[\alpha]_{X\setminus F}: \alpha\in \pi_X(\partial X')\cup \pi_X(\partial Z')\}).\]
Then we construct the hierarchy \(\widetilde{\HH}(X,Z)\) by replacing ``velcrot'' by ``equal'', and \(2\theta\) by \(\theta\) in \autoref{def: hierarchy}. Indeed, since \(\theta>0\) can be \emph{a posteriori} taken large enough, we can assume that it is large so that the hierarchy \(\widetilde{\HH}(X,Z)\) verifies the conditions needed to run the Bestvina--Bromberg--Fujiwara machinery. The modified projection distance \(\widetilde{d}_Y(X,Z)\) is defined as in \eqref{eq def: modified proj dist} with hierarchy \(\widetilde{\HH}(X,Z)\). 

We will now approximate \(\PY\) by \(\mathcal{P}_{R}(\widetilde{\Y})\).

\begin{lemma}\label{lem: surv proj}
Let \(F\) be any finite collection of points on \(S\) and \(X\) be a subsurface on \(S\) such that \(\partial X\cap F=\emptyset\). Suppose that \(x,y\in \C^\dagger(S)\) are two curves intersecting \(X\) essentially, then for any \(\alpha\in \pi_X(x)\) and \(\beta\in \pi_X(y)\) that are transverse, we have
\[d_{\C^s(X\setminus F)}\big([\alpha]_{X\setminus F},[\beta]_{X\setminus F}\big)\leq d_{\C^\dagger(X)}(\alpha,\beta),\]
where \([\alpha]_{S\setminus F},[\beta]_{S\setminus F}\) are the isotopy classes on \(S\setminus F\). The equality is attained if \(\alpha,\beta\) are in minimal position on \(S\setminus F\).
\end{lemma}
\begin{proof}
This is a direct consequence of \autoref{lem: dagger}. 
\end{proof}

The following lemma indicates that the definitions of two hierarchy are compatible:

\begin{lemma}\label{lem: inclusion hierarchy}
Let \(\widetilde{\Y}\) be as above. For any distinct elements \(X,X',Z'\in \widetilde{\Y}\), we have
\[|\widetilde{d}^\pi_X(X',Z')-d^\pi_X(X',Z')|\leq 24<\theta.\]
Therefore, \(\widetilde{\HH}(X,Z)\supset \HH(X,Z)\cap \widetilde{\Y}\times \widetilde{\Y}\).
\end{lemma}
\begin{proof}
The inequality follows immediately from \autoref{lem: surv proj} and \autoref{prop: bdd diam}. By the definition of \(\widetilde{\HH}(X,Z)\), the desired inclusion is a direct consequence of the above inequality.
\end{proof}

Now, we will have the following estimate for modified projection distances:

\begin{lemma}\label{lem: punctured dist estimate}
For any distinct non-velcrot \(X,Z\in \Y\) with boundaries intersecting transversely, there exists a subcollection \(\widetilde{\Y}\) as above such that \(X,Z\in \widetilde{\Y}\) and for any \(Y\in \widetilde{\Y}\setminus\{X,Z\}\), we have
\[\widetilde{d}_Y(X,Z)\leq d_Y(X,Z)+\theta.\]
\end{lemma}
\begin{proof}
Since the boundaries of \(X,Y,Z\) are intersecting transversely, there exists a finite collection of points \(F\subset S\) such that \(\partial X,\partial Y,\partial Z\) are pairwise in minimal position on \(S\setminus F\). Let \(\widetilde{\Y}\) be a subcollection as above that contains \(X,Y,Z\) and representatives of isotopy classes in \(\hh(X,Z)\) on \(S\setminus F\) so that the boundaries are in minimal positions. By \autoref{lem: inclusion hierarchy}, we have
\begin{align*}
    \widetilde{d}_Y(X,Z)&= \inf\{\widetilde{d}^\pi_Y(X',Z') : (X',Z')\in \widetilde{\HH}(X,Z)\}\\
    &\leq \inf\{d^\pi_Y(X',Z') : (X',Z')\in \widetilde{\HH}(X,Z)\}+\theta\\
    &\leq \inf\{d^\pi_Y(X',Z') : (X',Z')\in \HH(X,Z)\cap \widetilde{\Y}\times \widetilde{\Y}\}+\theta\\
    &=\inf\{d^\pi_Y(X',Z') : (X',Z')\in \HH(X,Z)\}+\theta\\
    &=d_Y(X,Z)+\theta,
\end{align*}
where the second last equality comes from the fact that the projection distance is minimised when \(\partial X'\) and \(\partial Z'\) are in minimal positions on \(S\setminus F\), \emph{i.e.} when \(X',Z'\in \widetilde{\Y}\), see \autoref{lem: surv proj}.
\end{proof}

The following proposition allows us to give a lower bound for vertices in \(\PY\). It turns out that the lower bound is also similar to the lower bound in the original Bestvina--Bromberg--Fujiwara construction, see \cite[Lemma 3.18]{bestvina2015constructing}.

\begin{proposition}
There exists \(K'>0\) sufficiently large such that for any \(X,Z\in \Y\), we have
\[\dP(X,Z)\geq |\Y_{K'}(X,Z)|_{\vel}\, .\]
\end{proposition}
\begin{proof}
If \(X,Z\) are velcrot, then there is nothing to prove. Suppose now that they are not. To deduce the desired result, it suffices to show that for any \(X,Z\in \Y\) with \(\partial X\) and \(\partial Z\) intersecting transversely, we have
\[\dP(X,Z)\geq |\Y_{K'}(X,Z)|_{\vel}+1.\]
Indeed, we can perturb \(X,Z\) a little to \(X',Z'\) so that their boundaries intersect transversely, and \(X,X'\), as well as \(Z,Z'\), are velcrot to each other, and by coarse triangle inequality \ref{IV}, we can conclude that \(|\Y_{K'}(X',Z')|_{\vel}\leq |\Y_{K_0}(X,Z)|_{\vel}\) for \(K'\gg K_0\).

Let \(K'\gg K_0\simt 5R\). Let \((Y_i)_{i=1}^n\subset \Y_{K'}(X,Z)\) that are pairwise non-velcrot and that every \(Y\in \Y_{K'}(X,Z)\) is velcrot to one of \(Y_i\)'s. Suppose in addition that they are pairwise transverse. Now, take \(F\subset S\) be a finite collection such that \(\partial X, \partial Z, \partial Y_i\) are pairwise in minimal position on \(S\setminus F\). Consider a subcollection \(\widetilde{\Y}\subset \Y\) containing \(Y_i\)'s and satisfying the conditions in \autoref{lem: punctured dist estimate}. We then have \(\widetilde{\Y}_{K_0}(X,Z)\supset (Y_i)_i\). Since \(K_0\st 5R\), by \cite[Lemma 3.18]{bestvina2015constructing}, we can conclude that
\[d_{\mathcal{P}_R(\widetilde{\Y})}(X,Z)\geq | \widetilde{\Y}_{K_0}(X,Z)|+1\geq n+1.\]
By \autoref{lem: punctured dist estimate} again, we can see that
\[\widetilde{\Y}_R(X,Z)\subset \Y_K(X,Z)\cap \widetilde{\Y}.\]
so if \(\dP(X,Z)=1\) in \(\PY\), they are also connected by an edge in \(\mathcal{P}_R(\widetilde{\Y})\). This implies that
\[\dP(X,Z)\geq d_{\mathcal{P}_R(\widetilde{\Y})}(X,Z)\geq n+1.\]
Now, taking the infimum over all collections \((Y_i)_{i=1}^n\subset \Y_{K'}(X,Z)\) as given at the beginning of this paragraph, we can conclude the desired result.
\end{proof}

Now, using arguments similar to \cite[Proposition 3.20]{bestvina2015constructing} and the fact that iterating pseudo-Anosov homeomorphisms relative to finitely many of points can make the projection distance arbitrarily large, we get:
\begin{corollary}
The quasi-tree \(\PY\) is unbounded. \QEDD
\end{corollary}

Finally, since the projection distance is \(\Homeo_0(S)\)-equivariant, {\it i.e.}
\[d^\pi_{g(Y)}\big(g(X),g(Z)\big)=d^\pi_Y(X,Z)\]
for any \(X,Y,Z\in \Y\) and any \(g\in \Homeo_0(S)\), we can deduce that \(\Homeo_0(S)\) also acts on \(\PY\) by automorphism (or equivalently, by isometries). Recall that an isometric action of a group \(G\) on a metric space \(X\) is \emph{cobounded} if a Hausdorff neighbourhood of a \(G\)-orbit contains \(X\).

\begin{proof}[Proof of \autoref{thm: cobdd action on quasi tree}]
Indeed, it suffices to take \(\Y\) consisting of only finitely many isotopy classes of non-sporadic essential subsurface, thus the isometric action of \(\Homeo_0(S)\) on \(\PY\) has only finite orbits, which, \emph{a fortiori}, implies that the action is cobounded. 
\end{proof}
\section{Blowing up}\label{sec-5}

In this section, we will modify the fine projection complex constructed in \autoref{sec-4} into a Gromov hyperbolic space in which the fine curve graphs of essential subsurfaces are quasi-isometrically embedded. As an application, we make use of the Bestvina--Fujiwara machinery to build quasi-morphisms for surface homeomorphisms supported on subsurfaces.

\subsection{Construction}\label{sec-5A}
Now we will define a new graph \(\CY\) as follows. The vertex set is the collection of ordered pairs \((X,\alpha)\), where \(X\in\Y\) and \(\alpha\in\mathcal{C}^\dagger(X)\). Let \(L\simt K\) and we may assume that \(K<L<2K\). We connect \((X,\alpha),(Y,\beta)\in \CY\) following the rules below:
\begin{enumerate}[label=(L\arabic*),topsep=0pt, itemsep=-1ex, partopsep=1ex, parsep=1ex ]
    \item \((X,\alpha)\) is connected to \((X,\beta)\) by an edge of length \(1\), if \(\alpha,\beta\) are both essential in \(X\) and are disjoint. \label{L1}
    \item \((X,\alpha)\) is connected to \((Y,\beta)\) by an edge of length \(L\), if \(X,Y\) are velcrot and \(\beta\in\pi_Y(\alpha)\). \label{L2}
    \item \((X,\alpha)\) is connected to \((Y,\beta)\) by an edge of length \(L\), if \(X,Y\) are not velcrot but \(\dP(X,Y)=1\), and if \(\alpha\in \pi_X(Y)\) and \(\beta\in \pi_Y(X)\). \label{L3}
\end{enumerate}
We equip \(\CY\) with the combinatorial distance given above, denoted \(d_{\CY}\). We also note that the condition \ref{L1} implies that for every \(X\in\Y\), the fine curve graph \(\mathcal{C}^\dagger(X)\) is an abstract subgraph of \(\CY\). Since \(\PY\) and \(\C^\dagger(X)\) for \(X\in \Y\) are connected, it is clear from the definition that \(\CY\) is also connected.

For simplicity, let us define for any essential subsurface \(X\subset S\) and any two finite collections \(A,B\subset \C^\dagger(S)\),
\[d^\pi_X(A,B)\coloneqq\diam_{\C^\dagger(X)}\left(\bigcup_{\alpha\in A,\, \beta\in B}\big(\pi_X(\alpha)\cup\pi_X(\beta)\big)\right).\]
If \(p=(Y,\alpha)\in \CY\), then we will define \(\pi_X(p)\coloneqq\pi_X(\alpha)\). Recall that if \(p=Y\in \Y\), we have defined \(\pi_X(p)\coloneqq\pi_X(\partial Y)\).

Now, we may give the first estimate of the distance between some points in \(\CY\):

\begin{proposition}\label{prop: fine curve graph bi-lip}
For any \(X\in\Y\), the inclusion \(\iota_X\colon\C^\dagger(X)\hookrightarrow\CY\) by \(\alpha\mapsto (X,\alpha)\) is bi-Lipchitz. The bi-Lipschitz constant does not depend on the choice of \(X\in \Y\).
\end{proposition}

\begin{proof}
For any \(\alpha,\beta\in \C^\dagger(X)\), note that
\[d^\dagger_X(\alpha,\beta)\geq d_{\CY}\big((X,\alpha),(X,\beta)\big)\]
as \(\mathcal{C}^\dagger(X)\) is a subgraph of \(\CY\). 

Conversely, let us consider the graph \(\Gamma(X)\) by collapsing all \(\{(Y,\alpha):\alpha\in \C^\dagger(Y)\}\) into a singleton, still denoted \(Y\), for every \(Y\in \Y\) not velcrot to \(X\). Let \(\pi\colon \CY\to \Gamma(X)\) be the collapsing map and equip \(\Gamma(X)\) with the combinatorial distance, then \(\pi\) is \(1\)-Lipschitz. For any \(\alpha,\beta\in \C^\dagger(X)\), let \((X,\alpha)=x_0,x_1,\dots,x_n=(X,\beta)\) be a geodesic segment in \(\Gamma(X)\) connecting \((X,\alpha)\) to \((X,\beta)\). Now we have the following situations (up to a change of order):
\begin{mycases}
    \item If both \(x_i,x_{i+1}\in \Gamma(X)\) represent a subsurface not velcrot to \(X\), then by the Coarse equality~\ref{II}, we conclude
    \[d^\pi_X(x_i,x_{i+1})\pt d_X(x_i,x_{i+1})<K<L=d_{\Gamma(X)}(x_i,x_{i+1}).\]
    \item If \(x_i=(X',\gamma)\) and \(x_{i+1}=(X',\gamma')\), where \(\gamma\cap \gamma'=\emptyset\) are two essential curves on \(X'\) for some \(X'\) velcrot (or even equal) to \(X\), then we first remark that \(\gamma,\gamma'\) necessarily intersect essentially with \(X\), and by \autoref{prop_disjoint_subsurface_proj}, one can also conclude
    \[d^\pi_X(\gamma,\gamma')\leq 12=12 \cdot d_{\Gamma(X)}(x_i,x_{i+1}).\]
    \item If \(x_i=(X',\gamma)\) and \(x_{i+1}=(X'',\gamma)\) for some \(\gamma'\in \pi_{X''}(\gamma')\) and \(X',X''\) velcrot (or equal) to \(X\). Then either \(\gamma\cap\gamma'=\emptyset\) or \(\gamma=\gamma'\). In either case, both \(\gamma\) and \(\gamma'\) have essential intersection with \(X\), by \autoref{prop_disjoint_subsurface_proj}, we have
    \[d^\pi_X(\gamma,\gamma')\leq 12<L=d_{\Gamma(X)}(x_i,x_{i+1}).\]
    \item If \(x_i\) represent a subsurface \(Y\in \Y\) and \(x_{i+1}=(X',\gamma)\) for some \(\gamma\in \C^\dagger(X')\) and \(X'\) velcrot (or equal) to \(X\), then either \(X',Y\) are velcrot and \(\gamma\) intersects \(Y\) essentially, or \(\gamma\in \pi_{X'}(Y)\). Note that as \(Y\) is not velcrot to \(X\), so \(\partial Y\) intersects \(X\) essentially. For the first case, let \(\gamma'\in \pi_Y(\gamma)\) such that either \(\gamma'=\gamma\), or \(\gamma'\subset Y\cap X'\) is a curve disjoint from \(\gamma\). Then \(\gamma',\gamma,\partial Y\) are contained in a geodesic segment in \(\C^\dagger(S)\) between \(\gamma\) and \(\partial Y\), and by \autoref{thm: bgit}, we have
    \[d^\pi_X(\gamma,Y)\leq M<L=d_{\Gamma(X)}(x_i,x_{i+1}).\]
    For the second case, we can see that \(\gamma\cap \partial Y=\emptyset\), then by \autoref{prop_disjoint_subsurface_proj}, we have
    \[d^\pi_X(\gamma,Y)\leq 12<L=d_{\Gamma(X)}(x_i,x_{i+1}).\]
\end{mycases}

Using the triangle inequality, we can deduce that
\[d^\dagger_X(\alpha,\beta)\leq 12\cdot d_{\Gamma(X)}\big((X,\alpha),(X,\beta)\big)\leq 12\cdot d_{\CY}\big((X,\alpha),(X,\beta)\big),\]
which completes the proof.
\end{proof}

\subsection{Geometric properties} For our convenience, let us introduce the following modified projection distance. Let \(\iota_X\colon \C^\dagger(X)\hookrightarrow \CY\) be as in \autoref{prop: fine curve graph bi-lip}. Consider \(y,z\in \CY\) and define \(d_X(y,z)\) in the following way:
\begin{enumerate}[label=(\arabic*),topsep=0pt, itemsep=-1ex, partopsep=1ex, parsep=1ex ]
    \item If \(y,z\notin \iota_X(\C^{\dagger}(X'))\) for any \(X'\) velcrot to \(X\), and if \(y=(Y,\alpha)\) and \(z=(Z,\beta)\), then \(d_X(y,z)\coloneqq d_X(Y,Z)\).
    \item If either \(y\) or \(z\) is contained in \(\iota_X(\C^\dagger(X'))\)  for some \(X'\) velcrot to \(X\) with \(y=(Y,\alpha)\) and \(z=(Z,\beta)\), then \(d_X(y,z)\coloneqq d^\pi_X(\{\alpha\}\cup\partial Y,\{\beta\}\cup\partial Z)\).
\end{enumerate}
In particular, if \(Z\) is not velcrot to \(X\) and \(y=(Y,\alpha)\) with \(\alpha\) intersecting \(X\) essentially, then we may write \(d_X(y,Z)\coloneqq d_X^\pi(\alpha,\partial Z)\). 

\begin{remark}
Here in the case where one of \(y,z\) is contained in \(\iota_X(\C^\dagger(X'))\)  for some \(X'\) velcrot to \(X\), the reason why we have set \(d_X(y,z)\) to be \(d^\pi_X(\{\alpha\}\cup\partial Y,\{\beta\}\cup\partial Z)\) instead of \(d^\pi_X(\alpha,\beta)\) is purely technical, since \(\alpha\) or \(\beta\) may not intersect \(X\) essentially. However, as \(\alpha\) (resp. \(\beta\)) is disjoint from \(\partial Y\) (resp. \(\partial Z\)), which intersects \(X\) essentially, we may consider that the ``projection to \(X\)'' of \(\alpha\) (resp. \(\beta\)) is roughly \(\pi_X(Y)\) (resp. \(\pi_X(Z)\)). See \autoref{prop: projection consistent} for a further justification.
\end{remark}

\begin{lemma}\label{lem: coarse eq for curves}
Let \(y,z\in \CY\) and \(X\in \Y\). Suppose that \(y=(Y,\alpha)\) and \(z=(Z,\beta)\). Then \(d_X(y,z)\simt d^\pi_X(\{\alpha\}\cup\partial Y,\{\beta\}\cup\partial Z)\).
\end{lemma}
\begin{proof}
If none of \(Y,Z\) is velcrot to \(X\), then together with coarse equality~\ref{II}, we have
\[d_X(y,z)=d_X(Y,Z)\simt d^\pi_X(Y,Z).\]
But, as \(\alpha\cap \partial Y=\emptyset\) and \(\beta\cap \partial Z=\emptyset\), by \autoref{prop_disjoint_subsurface_proj}, we can conclude that
\[d_X(y,z)\simt d^\pi_X(Y,Z)\simt d^\pi_X(\{\alpha\}\cup\partial Y,\{\beta\}\cup\partial Z).\]
If one of \(Y,Z\) is velcrot to \(X\), then \(d_X(y,z)=d^\pi_X(\{\alpha\}\cup\partial Y,\{\beta\}\cup\partial Z)\).
\end{proof}

The following gives a coarse triangle inequality for the blown-up fine projection complex setting:

\begin{lemma}
For any \(x,y,z\in \CY\) and any \(Y\in \Y\), we have
\begin{align}
d_Y(x,y)+d_Y(y,z)\st d_Y(x,z).\label{eq: coarse triangle ineq}
\end{align}
\end{lemma}
\begin{proof}
It follows from coarse triangle inequality \ref{IV} and \autoref{lem: coarse eq for curves}.
\end{proof}

As we have defined the modified projection distance for the vertices in \(\CY\), we may also define \(\Y_K(x,z)\), for \(K>0\) and \(x,z\in \CY\), by the collection of all \(Y\in \Y\) such that \(d_Y(x,z)>K\). Moreover, for \(x\in \CY\) and \(Z\in \Y\), we can define \(\Y_K(x,Z)\) in a similar way. We remark that for \(x\in\iota_X(\C^\dagger(X))\) and \(z\in \iota_Z(\C^\dagger(Z))\), the set \(\Y_K(x,z)\) might also contain \(X\) and \(Z\).

\begin{lemma}\label{lem: C dist control}
Let \(X_0,X_1\) be two non-velcrot vertices in \(\PY\) such that \(\dP(X_0,X_1)=1\). Assume that \(x_0\in \iota_{X_0}(\pi_{X_0}(X_1))\subset\iota_{X_0}(\C^\dagger(X_0)\) and \(x_1\in \iota_{X_1}(\pi_{X_1}(X_0))\subset\iota_{X_1}(\C^\dagger(X_1)\). Let \(W\in \Y\) be also a vertex in \(\PY\) and \(w\in\iota_W(\C^\dagger(W))\) such that \(d_{\CY}(x_i,w)\geq 2L\) for \(i=0,1\). Then either
\[d_W(x_0,x_1)\simt 0\]
or for \(i=0,1\),
\[d_W(x_i,w)\st L\,.\]
If \(X_0,X_1\) are velcrot, \(x_0=(X_0,\alpha)\), \(x_1=(X_1,\alpha)\) for some \(\alpha\in \C^\dagger(X_0)\cap\C^\dagger(X_1)\), and let \(w,W\) be as above, then  \(d_W(x_0,x_1)\simt 0\).
\end{lemma}
\begin{proof}
Let us first prove the last part of the statement. Let \(X_0,X_1\) be velcrot and \(x_0=(X_0,\alpha)\), \(x_1=(X_1,\alpha')\) for some \(\alpha'\in \pi_{X_1}(\alpha)\). If \(W\) is velcrot to either \(X_0\) or \(X_1\), then either \(\alpha\) or \(\alpha'\) also intersects \(W\) essentially, which in turn by \autoref{prop_disjoint_subsurface_proj} implies that
\[d_W(x_0,x_1)=d^\pi_W(\{\alpha\}\cup \partial X_0,\{\alpha'\}\cup\partial X_1)\leq M\simt 0\,,\]
since we also have \(\alpha\cap\alpha'=\alpha\cap \partial X_0=\alpha'\cap \partial X_1=\emptyset\). Otherwise, \(W\) is not velcrot to \(X_0\) nor \(X_1\). Then \(\partial X_0\) and \(\partial X_1\) both intersect \(W\) essentially, and by the coarse equality \ref{II} and \autoref{lem: velcrot small proj}, we can conclude
\[d_W(x_0,x_1)=d_W(X_0,X_1)\leq d^\pi_W(X_0,X_1)\leq M\simt 0.\]

Now, let us assume that \(X_0,X_1\) are not velcrot. If \(W\) is velcrot to either \(X_0\). Suppose that \(x_0=(X_0,\alpha)\) and \(x_1=(X_1,\beta)\). Note that \(\alpha\in \pi_{X_0}(X_1)\) intersecting \(W\) essentially and, in particular, \(\alpha\cap \partial X_1=\emptyset\). As \(\beta\in\C^\dagger(X_1)\), we also have \(\beta\cap \partial X_1=\emptyset\). Then by \autoref{prop_disjoint_subsurface_proj} and the triangle inequality
\[d_W(x_0,x_1)=d^\pi_W(\alpha,\{\beta\}\cup\partial X_1)\leq d^\pi_W(\alpha,\partial X_1)+d^\pi_W(\beta,\partial X_1)\leq 24\simt 0.\]
The case where \(W\) is velcrot to \(X_1\) can be shown similarly.

Finally, if \(X_0,X_1,W\) are pairwise non-velcrot, then the proof is the same as in \cite[Lemma~4.5]{bestvina2015constructing}. Nevertheless, we provide the full proof here for the convenience of the reader.

If \(\dP(X_0,W)\geq 2\) or \(\dP(X_1,W)\geq 2\), then by \autoref{rem: non-velcrot PC dist control}, we can see that \(d_W(x_0,x_1)=d_W^\pi(X_0,X_1)\leq d_W(X_0,X_1)\simt 0\).

Now, it remains the case where \(\dP(X_0,W)=\dP(X_1,W)=1\). We observe that if \(d_{X_0}(X_1,W)\geq \Theta\) for \(\Theta>0\) from \autoref{thm: BBF prerequi}, then by the inequality on triples~\ref{V}, we get \(d_W(X_0,X_1)\simt 0\). The same estimate holds when \(d_{X_1}(X_0,W)\geq \Theta\).

So we may assume in the following that \(d_{X_0}(X_1,W),d_{X_1}(X_0,W)<\Theta\simt 0\). Suppose that \(x_0=(X_0,\alpha)\), \(x_1=(X_1,\beta)\) and \(w=(W,\gamma)\). Consider a path made up of a path in \(\C^\dagger(X_0)\) connecting \(\alpha\) to \(\pi_{X_0}(\gamma)\), an edge from \(\pi_{X_0}(W)\) to \(\pi_W(X_0)\), and a path in \(\C^\dagger(W)\) connecting \(\pi_W(X_0)\) to \(\gamma\). This would yield the inequality
\[d_{\CY}(x_0,w)\leq d_{X_0}^\pi (\alpha,\gamma)+L+d^\pi_W(\alpha,\gamma).\]
Since \(\alpha\in \pi_{X_0}(X_1)\), we have \(d_{X_0}^\pi (\alpha,\gamma)\leq d^\pi_{X_0}(X_1,\gamma)+12\) by \autoref{prop: bdd diam}. By \autoref{lem: coarse eq for curves}, we can conclude
\[2L\leq d_{\CY}(x_0,w)\pt d_{X_0}(X_1,w)+L+d_W(x_0,w)\simt L+d_W(x_0,w),\]
which gives the estimate \(d_W(x_0,w)\st L\). The same bound also holds for \(d_W(x_1,w)\).
\end{proof}

\begin{lemma}\label{lem: fork guarding}
For \(K>\) sufficiently large the following holds. Let \(x_0\) and \(x_1\)
be adjacent vertices in \(\CY\) and let \(Y\) be a vertex in \(\PY\) such that \(d_{\CY}\big(x_i,\iota_Y(\C^\dagger(Y))\big)\geq 4L\). If \(W\) is a guard for \(Y\) with \(W \in \Y_{K/2}(x_0, Y )\) and \(W \notin \Y_{K/2}(x_1, Y )\), then there exists a guard \(W '\) for \(Y\) with \(W' \in \Y_{K/2}(x_1, Y )\) and \(W \in \Y_{\Theta}(W', Y )\).
\end{lemma}
\begin{proof}
Suppose that \(X_0,X_1\in \PY\) with \(x_i=(X_i,\alpha_i)\) and \(\alpha_i\in\C^\dagger(X_i)\) for \(i=0,1\). We start with the remark that \(Y,W\) are not velcrot, nor are \(Y,X_i\) for \(i=0,1\). The latter is a consequence of the assumption \(d_{\CY}\big(x_i,\iota_Y(\C^\dagger(Y))\big)\geq 4L\).

If \(X_0=X_1\) while \(W\) is not velcrot to them, then \(d_W(x_0,Y)=d_W(x_1,Y)\) and the lemma is vacuous. Otherwise, if \(W\) is velcrot (or equal) to \(X_0=X_1\), then we have
\begin{align*}
    4L & \leq d_{\C^\dagger(\Y)} \big(x_i,\iota_Y(\C^\dagger(Y))\big) \\
       & \leq d_{\C^\dagger(\Y)} \big(x_i,\pi_Y(W)\big)\\
       & \leq d_{\C^\dagger(\Y)}(x_i,\pi_W(Y)) +L\\
       & \leq d^\pi_W(\alpha_i, Y)+2L,
\end{align*}
which implies that \(d_W(x_i,Y)=d^\pi_W(\alpha_i, Y)\geq 2L\), resulting \(d_W(x_1,Y)>K/2\) as \(L\simt K\). So in this case, the lemma is also vacuous.

Now assume that \(X_0\neq X_1\). Noticing that
\[d_{\CY}\big(x_i,\iota_W(\pi_W(Y)\big)\geq d_{\CY}\big(x_i,\iota_Y(\C^\dagger(Y))\big)-L\geq 3L,\]
we can apply \autoref{lem: C dist control} to \(w\in \iota_W(\pi_W(Y))\), resulting either \(d_W(x_0,x_1)\simt 0\) or \(d_W(x_i,w)\st L\) for \(i=0,1\). If \(d_W(x_1,w)\st L> K\), then \(d_W(x_1,w)>K/2\), contradicting our assumption. Hence, we have \(d_W(x_0,x_1)\simt 0\).

We now claim that \(W\in\Y_{\Theta}(x_1,Y)\). If \(W\) is not velcrot to \(X_1\), then by coarse triangle inequality~\eqref{eq: coarse triangle ineq}, we have
\[d_W(x_1,Y)\simt d_W(x_0,x_1)+d_W(x_1,Y)\st d_W(x_0,Y)>K/2,\]
which forces \(d_W(x_1,Y)>\Theta\) for sufficiently large \(K>0\). If \(W\) is velcrot to \(X_1\), then \(d_W(X_1,Y)=d^\pi_W(\alpha_1,Y)\) and \(d^\pi_W(\alpha_1,x_0)=d^\pi_W(\alpha_1,\{\alpha_0\}\cup \partial X_0)\simt 0\), which again by triangle inequality~\eqref{eq: trian ineq proj dist}, gives
\[d^\pi_W(\alpha_1,Y)\simt d^\pi_W(\alpha_1,Y)+d^\pi_W(\alpha_1,x_0)\geq d^\pi_W(X_0,Y)\simt d_W(X_0,Y)>K/2.\]
This also indicates that \(d_W(x_1,Y)>\Theta\) for sufficiently large \(K>0\).

Now, let us check that \(\Y_K(x_1,Y)\) is not empty. Indeed, otherwise they are adjacent in \(\PY\) and we have
\[d_{\CY}\big(x_1,\iota_Y(\C^\dagger(Y))\big)\leq d^\pi_{X_1}(x_1,Y)+L\simt d_{X_1}(x_1,Y)+L\leq K+L<4L,\]
a contradiction to our assumption. Since \(W\) is a guard for \(Y\), every element in \(\Y_K(x_1,Y)\) should be less than \(W\) in \(\Y_\Theta(x_1,Y)\) after a slightly generalised version of \autoref{lem: K/2 empty means guard}, of which the proof is a direct adaption of the proof of \autoref{lem: K/2 empty means guard}. Therefore, there exist elements of \(\Y_K(x_1,Y)\), not velcrot to \(W\), that are less than \(W\) in \(\Y_\Theta(x_1,Y)\). Now, applying the same arguments in the proof of \autoref{lem: barrier induction} yields the desired result.
\end{proof}

Similar to what we have done in \(\PY\), let us define for a path \(\{x_0,x_1,\dots,x_k\}\) in \(\CY\) a \emph{barrier} \(Y\in \Y\) between this path and \(Z\in\Y\) by an element such that \(Y\in \Y_\Theta(x_i,Z)\) for \(0\leq i\leq k\). Note that it is possible that \(x_i\in\iota_Y(\C^\dagger(Y))\).

Similarly, we have a slightly generalised version of \autoref{lem: barrier implies small proj}:
\begin{lemma}
If there is a barrier \(Y\in\Y\) between a path \(\{x_0, \dots , x_k\}\) in \(\CY\) and \(Z\in \Y\), then \(d_Z(x_i,x_j)\leq \Theta\) for all \(i,j\).
\end{lemma}
\begin{proof}
Let \(x_i=(\alpha_i,X_i)\). We will first see that \(X_i\) is not velcrot to \(Z\). Indeed, suppose for contradiction that \(X_i,Z\) are velcrot. If \(X_i\) is not velcrot to \(Y\), then
\[\Theta<d_Y(x_i,Z)=d_Y(X_i,Z)=0, \]
contradiction. Then \(X_i\) is velcrot to \(Y\) and, therefore, \(\alpha_i\) intersects \(Y\) essentially. Now, let \(\gamma\in \pi_Z(\alpha_i)\) be such that some subarcs of \(\gamma\) is contained in a very small neighbourhood of \(\alpha_i\), forcing that \(\gamma\) also intersects \(Y\) essentially. Note that \(\alpha_i\cap \gamma=\gamma\cap \partial Z=\emptyset\). By \autoref{prop_disjoint_subsurface_proj} and triangle inequality~\eqref{eq: trian ineq proj dist}, we have
\[d_Y(x_i,Z)=d^\pi_Y(\alpha,Z)\leq d^\pi_Y(\alpha,\gamma)+d^\pi_Y(\gamma,Z)\leq 2M<\theta<\Theta,\]
also a contradiction. 

Now, as \(X_i\) is not velcrot to \(Z\), we have \(d_Z(x_i,x_j)=d_Z(X_i,X_j)\) for all \(i,j\). So we need only to show that \(d_Z(X_i,X_j)\leq \Theta\). Note that \(\{X_0,\dots,X_k\}\) also yield a path in \(\PY\). If neither \(X_i\) nor \(X_j\) is velcrot to \(Y\), then the desired result follows from \autoref{lem: barrier implies small proj}. If exactly one of the two is velcrot to \(Y\), say \(X_i\), then \(d_Z(X_i,Y)=0\), while \(d_Z(X_j,Y)<\theta\) by inequality on triples~\ref{V}. Hence, as \(Z,Y\) are not velcrot, the coarse triangle inequality~\eqref{eq: coarse triangle ineq} yields
\[d_Z(X_i,X_j)\pt d_Z(X_i,Y)+d_Z(X_j,Y)<\theta,\]
giving \(d_Z(X_i,X_j)<\Theta\). Similarly, if both \(X_i,X_j\) are velcrot to \(Y\), then we also have
\[d_Z(X_i,X_j)\pt d_Z(X_i,Y)+d_Z(X_j,Y)=0,\]
and, thus, \(d_Z(X_i,X_j)<\Theta\).
\end{proof}

Now using the exact arguments of the proof of \autoref{prop: barrier} while replacing the path in \(\PY\) by a path in \(\CY\), \autoref{lem: PC dist control} by \autoref{lem: C dist control} and \autoref{lem: barrier induction} by \autoref{lem: fork guarding}, we can conclude:
\begin{proposition}\label{prop: barrier C}
The following holds if \(K>0\) is sufficiently large. Let \(\{x_0, x_1, \dots , x_k\}\) be a path in \(\CY\) and \(Z\) an element in \(\Y\) such that \(d_{\CY}(x_i,\iota_Z(\C^\dagger(Z)) ) \geq 4L\) for all \(i\). Then there is a barrier \(W\) between the path and \(Z\). In particular, \(d_Z (x_0, x_i) <\Theta\) for all \(i\). \QEDD
\end{proposition}

\subsection{Hyperbolicity} 

In this subsection, we will show that \(\CY\) is Gromov hyperbolic, using Guessing Geodesics Lemma.

Define the map \(\sigma\colon \CY\to \PY\) by \((X,\alpha)\mapsto X\).

\begin{lemma}\label{lem: Y'-Z'-Z}
Let \(y=(Y,\eta)\) be a vertex in \(\CY\) and \(Z\in \Y\) not velcrot to \(Y\in \Y\). Suppose that \(z\) is one nearest point in \(\iota_Z(\C^\dagger(Z))\) to \(y\). Let \(\ell\) be a geodesic connecting \(y\) to \(z\). If the vertex on \(\ell\) next to \(y\) is in \(\iota_{Z'}(\C^\dagger(Z'))\) with \(Z,Z'\) velcrot, then \(d_{\CY}(y,z)=2L\) and \(d_Z(y,z)\simt 0\).
\end{lemma}
\begin{proof}
Since \(Z,Z'\) are velcrot, by \autoref{prop: witness}, we can find \(W\subset Z\cap Z'\) isotopic to both \(Z\) and \(Z'\). Let \(z'\coloneqq(Z',\alpha)\) be the vertex on \(\ell\) next to \(y\) is in \(\iota_{Z'}(\C^\dagger(Z'))\). By our assumption, we have
\[d_{\CY}(y,z)\leq d_{\CY}(y,z')+d_{\CY}(z',z)\leq 2L.\]

If \(Y,Z'\) are velcrot and \(\alpha\in \pi_Z(\eta)\), then \(\eta\) also intersects \(W\) essentially, so we may choose \(\alpha\subset W\). In turn, the path \(y\rightsquigarrow z'\rightsquigarrow (Z,\alpha)\) gives a path of length \(2L\) from \(y\) to \(\iota_Z(\C^\dagger(Z))\). This indicates that we can take \(z=(Z,\alpha)\) and that \(d_{\CY}(y,z)=2L\). Moreover, as \(\partial Y\cap \eta=\eta\cap \alpha=\emptyset\) and \(\partial Y\) intersects \(Z\) essentially, we can conclude
\[d_Z(y,z)=d^\pi_Z(Y,\alpha)\leq d^\pi_Z(Y,\eta)+d^\pi_Z(\eta,\alpha)\leq 2M\simt 0,\]
by triangle inequality~\eqref{eq: trian ineq proj dist} and \autoref{prop_disjoint_subsurface_proj}.

If \(Y,Z'\) are not velcrot, then the same result can be concluded via replacing \(\eta\) by \(\partial Y\) and repeating the discussion above.
\end{proof}

\begin{proposition}\label{prop: projection consistent}
Let \(x\) be a vertex in \(\CY\) and \(Z\in \Y\). Suppose that \(z\) is one nearest point in \(\iota_Z(\C^\dagger(Z))\) from \(x\). Then
\[d_Z(x,z)\pt 3K.\]
\end{proposition}
\begin{proof}
We first assume that \(d_{\CY}(x,z)\geq 4L\). Then there exists a last point \(y\in \CY\) on a geodesic connecting \(z\) to \(x\) such that \(d_{\CY}\big(y,\iota_Z(\C^\dagger(Z))\big)\geq 4L\). By \autoref{prop: barrier C}, we have \(d_Z(x,y)<\Theta\).

Note that a path in \(\CY\) of length at most \(kL-1\) is sent by \(\sigma\) to a path in \(\PY\) of length at most \(k-1\). By our choice of \(y\), we have \(d_{\CY}(y,z)\leq 5L-1\). Therefore, the geodesic from \(y\) to \(z\) in \(\CY\), denoted \(\ell\), will be mapped to a path \(\sigma(\ell)\) of length at most \(4\) in \(\PY\). As \(Y\coloneqq \sigma(y)\) is not velcrot to \(Z\), so the path we obtained in \(\PY\) is of length at least \(1\).

Let \(Z'\) be the last vertex in \(\sigma(\ell)\subset \PY\) before \(Z\). We claim that it is the only vertex in \(\sigma(\ell)\) that can be velcrot to \(Z\). Indeed, let \(y'\coloneqq(Y',\eta)\in \ell\subset \CY\) be the last vertex such that \(\sigma(y')\) is not velcrot to \(Z\) and \(w=(W,\beta)\) be the next vertex in \(\ell\) with \(W\) velcrot to \(Z\). Then \(d_{\CY}\big(y',\iota_Z(\C^\dagger(Z))\big)\leq 2L\), this is because one can find a path \(y'\rightsquigarrow w \rightsquigarrow (Z,\pi_Z(\beta))\) of length \(2L\) in \(\CY\) from \(y'\) to \(\iota_Z(\C^\dagger(Z))\). In turn, this forces \(\dP(Y',Z)\leq 2\), so either \(W=Z'\) or \(W=Z\).

Since \(\dP(Y,Z')\leq 3\) and no vertex on \(\sigma(\ell)\) between \(Y\) and \(Z'\) is velcrot to \(Z\), by coarse triangle inequality~\ref{IV}, we have
\(d_Z(Y,Z')\pt 3K\). If \(Z,Z'\) are not velcrot, then \(z\in \pi_Z(Z')\), so \(d_Z(Y,Z')=d_Z(y,z)\pt 3K\). Since \(d_Z(x,y)\simt 0\), by coarse triangle inequality~\ref{IV} again, we get \(d_Z(x,z)\pt 3K\) as desired. If \(Z,Z'\) are velcrot, then setting \(y'=(Y',\eta)\) as above, we are in the situation of \autoref{lem: Y'-Z'-Z} and \(d_Z(y,z)\pt d_Z(y,y')+d_Z(y',z)\simt d_Z(y,y')\), after applying coarse triangle inequality~\ref{IV}. However, as \(\dP(Y,Y')\leq 2\), we will have \(d_Z(y,y')=d_Z(Y,Y')\pt 2K\) via applying coarse triangle inequality~\ref{IV} as in the beginning of this paragraph, which further yields
\begin{align*}
d_Z(x,z) &\pt d_Z(x,y)+d_Z(y,z)\\
& \simt d_Z(x,y)+d_Z(y,y')\\
&\simt d_Z(y,y')\\
&\pt 2K<3K.
\end{align*}

Finally, for the situation where \(d_{\CY}\big(y,\iota_Z(\C^\dagger(Z))\big)\leq 4L\), it suffices to discuss in the same way as the previous three paragraphs by replacing \(x\) in lieu of \(y\) and we can obtain the same upper bound as \(d_Z(x,z) \pt 3K\).
\end{proof}

\begin{corollary}\label{cor: projection consistent}
For every \(Z\in \Y\) and \(Y\in \Y\) not velcrot to \(Z\), the nearest point projection \(\iota_Y(\C^\dagger(Y))\to \iota_Z(\C^\dagger(Z))\) is in a uniform neighbourhood of the bounded set \(\iota_Z(\pi_Z(Y))\).
\end{corollary}
\begin{proof}
Let \(y=(Y,\alpha)\in \iota_Y(\C^\dagger(Y))\) and \(z=(Z,\beta)\in \iota_Z(\C^\dagger(Z))\) be the image of \(y\). By \autoref{lem: coarse eq for curves} and \autoref{prop: projection consistent}, we have
\[3K\st d_Z(y,z)\simt d_Z^\pi(\partial Y,\beta),\]
which implies that the distance between \(z\) and \(\iota_Z(\pi_Z(Y))\) is uniformly bounded.
\end{proof}

\begin{proposition}\label{prop: path rigid}
Let \(x,z\in \CY\), \(X=\sigma(x)\), and \(Z=\sigma(z)\). If \(Y\in \Y_\Theta(x,z)\), then any path from \(x\) to \(z\) in \(\CY\) contains a vertex \(w\in\CY\) such that
\begin{itemize}
    \item \(d_{\CY}(w,\iota_Y(\C^\dagger(Y)))<4L,\)
    \item \(d_Y(x,w)\pt K,\)
\end{itemize}
and it follows that \(d_{\CY}\big(w,\iota_Y(\pi_Y(x))\big)\pt 4L+4K\). A similar inequality holds for \(z\) in place of \(x\).
\end{proposition}
\begin{proof}
By \autoref{prop: barrier C}, every path from \(x\) to \(z\) must intersect the \(4L\)-neighbourhood of \(\iota_Y(\C^\dagger(Y))\) for \(Y\) not velcrot to \(X\) or \(Z\). Let \(\ell\) be an arbitrary path from \(x\) to \(z\) in \(\CY\). But if \(Y\) is velcrot, then the situation becomes trivial. Now, let \(w\) be the first element on \(\ell\) with \(d_{\CY}(w,\iota_Y(\C^\dagger(Y)))<4L\). There is nothing to prove if \(w=x\), so we may also assume that \(w'\) is the vertex preceding \(w\). Suppose that \(w=(W,\alpha)\) and \(w'=(W',\alpha')\). Then \(W\) and \(W'\) are either adjacent in \(\PY\), or \(W=W'\). Then we have the following situations:
\begin{mycases}
    \item If \(W,W',Y\) are pairwise non-velcrot, then \(d_Y(w,w')=d_Y(W,W')<K\).
    \item If \(W,W'\) are not velcrot but \(Y\) is velcrot to one of them, say to \(W\), then \(d_Y(w,w')=d^\pi_Y(\alpha,W')\) and \(\alpha\in \pi_W(W')\). By the definition of the subsurface projection, \(\alpha\) is either disjoint from \(\partial W'\) or contained in \(\partial W\), so \(d^\pi_Y(\alpha,W')<M\pt K\) by \autoref{prop_disjoint_subsurface_proj}.
    \item If \(W,W'\) are velcrot and \(Y\) is velcrot to at least one of them (can be both of them), then \(d_Y(w,w')=d^\pi_Y(\alpha,\alpha')\), whereas \(\alpha\cap \alpha'=\empty\) or \(\alpha=\alpha'\), which further implies that \(d_Y(w,w')=d^\pi_Y(\alpha,\alpha')<M\pt K\) by \autoref{prop_disjoint_subsurface_proj}.
    \item If \(W,W'\) are velcrot but \(Y\) is not velcrot to them, then by \autoref{lem: velcrot small proj} and coarse equality~\ref{II}, we have \(d_Y(w,w')=d_Y(W,W')\simt d_Y^\pi(W,W')<M\pt K\). 
\end{mycases}
Hence, we can conclude that \(d_Y(w,w')\pt K\). Note that \(d_Y(x,w')\simt 0\) by \autoref{prop: barrier C}. So by coarse triangle inequality, we have \(d_Y(w,x)\pt K\).

Now, let \(\widetilde{w}\in \CY\) be a nearest point from \(w\) to \(\iota_Y(\C^\dagger(Y))\). We have by our assumption that \(d_{\CY}(w,\widetilde{w})<4L\). By \autoref{prop: projection consistent}, we see that \(d_Y(\widetilde{w},w)\pt 3K\). Now, the triangle inequality yields
\begin{align*}
d_{\CY}\big(w,\iota_Y(\pi_Y(x))\big)&\leq d_{\CY}(w,\widetilde{w})+d_{\CY}\big(\widetilde{w},\iota_Y(\pi_Y(x))\big)\\
&\leq d_{\CY}(w,\widetilde{w})+d^\pi_Y(\widetilde{w},w)+d_Y^\pi(w,x)\\
&= d_{\CY}(w,\widetilde{w})+d_Y(\widetilde{w},w)+d_Y(w,x)\\
&\pt 4L+3K+K=4L+4K.
\end{align*}
This gives the desired coarse upper bound.
\end{proof}

\begin{definition}[Standard path]\label{def: standard path}
Let \(x=(X,\alpha)\) and \(y=(Y,\beta)\) be two arbitrary vertices in \(\CY\). If \(X,Y\) are not velcrot, then \emph{standard path} between \(x,y\) is a path passing through \(\iota_X(\C^\dagger(X))\), \(\iota_{W_i}(\C^\dagger({W_i}))\), and \(\iota_Y(\C^\dagger(Y))\) in the natural order, where \(W_i\in \Y_K(X,Y)\) are pairwise non-velcrot such that any \(W\in \Y_K(X,Y)\) is velcrot to one of \(W_i\)'s, and within each \(\iota_{W}(\C^\dagger({W}))\) the standard path is the \(\iota_W\)-image of a geodesic segment in \(\C^\dagger(W)\). If \(X,Y\) are velcrot, then a \emph{standard path} between \(x\) and \(y\) is a path consisting of the \(\iota_X\)-image of a geodesic segment in \(\C^\dagger(X)\) that connects \(\alpha\) to some \(\gamma\in \pi_X(\beta)\) and an edge between \((X,\gamma)\) and \((Y,\beta)\) if \(X\neq Y\). We will denote by \(\LL(x,y)\) the union of all standard paths between \(x\) and \(y\).
\end{definition}

In the following, we will show that these standard paths serve as guessing geodesics for the application of Guessing Geodesics Lemma.

\begin{lemma}\label{lem: stand path proj closed}
There exists \(K'>0\) sufficiently large such that the following holds. Let \(x,z\in \CY\) and \(Y\in \Y_K(x,z)\cup\{\sigma(x),\sigma(z)\}\). If \(\ell\subset\CY\) is a standard path between \(x\) and \(z\) such that \(\ell \cap \iota_Y(\C^\dagger(Y))\neq \emptyset\), then there is a geodesic segment \([\alpha,\beta]\subset \C^\dagger(Y)\) with \(\iota_Y([\alpha,\beta])=\ell\cap \iota_Y(\C^\dagger(Y))\) and \(d_Y^\pi(x,\alpha),d_Y^\pi(\beta,z)\pt K'\).
\end{lemma}
\begin{proof}
If \(\sigma(x),\sigma(z)\) are velcrot, then the desired result follows directly from the definition. Hence, we may assume in the following that \(\sigma(x),\sigma(z)\) are not velcrot.

Note that \(\ell \cap \iota_Y(\C^\dagger(Y)\) is a path between \((Y,\alpha)\) and \((Y,\beta)\).

Let \(v=(Y,\alpha)\) and let \(v'=(Y',\gamma')\) be the vertex on \(\ell\) immediately preceding \(v\). Hence, \(\gamma'\in \pi_{Y'}(Y)\) and \(\alpha\in \pi_Y(Y')\). We claim that \(d_Y^\pi(x,v)=d_Y^\pi(x,\alpha)\pt K'\). For \(d_Y^\pi(\beta,z)\), it is a symmetric case. Suppose for contradiction that \(d_Y^\pi(x,v)> K'\) with \(K'\) sufficiently large. Then we further claim that \(\ell'\coloneqq \ell\cap \LL(x,v')\) must intersect \(\iota_Y(\C^\dagger(Y))\), which is a contradiction after our definition of a standard path. By \autoref{lem: coarse eq for curves}, as \(K'<d_Y^\pi(x,z)\simt d_Y(x,z)\) and \(K'\gg K\), we can conclude that \(Y\in \Y_K(x,z)\). By \autoref{prop: path rigid}, we can find \(v_0,w_0\in \ell'\) such that \(d_{\CY}\big(v_0, \iota_Y(\pi_Y(x))\big)\pt 4L+4K\) and \(d_{\CY}\big(w_0, \iota_Y(\pi_Y(v'))\big)\pt 4L+4K\). In particular, \(d_{\CY}(v_0,w_0)\pt 8L+8K+d_Y(x,v')\). If \(\ell'\) is disjoint from \(\iota_{Y''}(\C^\dagger(Y''))\) for any \(Y''\) velcrot to \(Y\), then we can estimate that the number of \(W\in\Y\) such that \(\ell'\) passes through \(\iota_W(\C^\dagger(W))\) is at least 
\[\frac{d^\pi_Y(x,v')}{K+2\theta}-1,\]
as the diameter of the projections to \(Y\) of the union of two consecutive non-velcrot \(W\)'s is at most \(K+2\theta\) by \autoref{prop: mod dist coarse to proj dist}. Thus, the number of edges of length \(L\) that the subsegment of \(\ell'\) between \(v_0\) and \(w_0\) passes through is at least \(d^\pi_Y(x,v')/(K+2\theta)\), and we have the inequality
\[\frac{Ld^\pi_Y(x,v')}{K+2\theta}\leq d_{\CY}(v_0,w_0)\pt 8L+8K+d^\pi_Y(x,v'),\]
which is contradictory if \(0<K'<d^\pi_Y(x,v')\) is sufficiently large as \(L/(K+2\theta)>1\). Hence, the subsegment of \(\ell'\) between \(v_0\) and \(w_0\) passes through \(\iota_{Y''}(\C^\dagger(Y''))\) for some \(Y''\) velcrot to \(Y\). But by our definition of a standard path, \(Y''\) has to be \(Y\), which will also contradict the assumption on \(v'\). Therefore, \(d^\pi_Y(x,v')< K'\). As \(v=(Y,\alpha)\) and \(v'=(Y',\gamma')\) with \(\alpha\in \pi_Y(Y')\), so \(\gamma'\cap \partial Y'=\alpha\cap \partial Y'=\emptyset\). By \autoref{prop_disjoint_subsurface_proj} and \eqref{eq: trian ineq proj dist}, as \(Y',Y\) are not velcrot, we can also conclude that 
\[d^\pi_Y(x,v)\leq d^\pi_Y(x,v')+d^\pi_Y(v',v)\leq d^\pi_Y(x,v')+d^\pi_Y(\gamma',Y')+d^\pi_Y(Y',\alpha)<K+24\pt K'.\]
This completes the proof.
\end{proof}

To show that \(\CY\) is Gromov hyperbolic, we will make use of the Guessing Geodesics Lemma (\autoref{prop: ggl}). 

\begin{theorem}\label{thm: CY is hyperbolic}
The space \(\CY\) is \(\delta\)-hyperbolic and for each \(X\in \Y\), the fine curve graph \(\C^\dagger(X)\) is embedded in \(\C^\dagger(\Y)\) via a bi-Lipschitz map \(\iota_X\).
\end{theorem}
\begin{proof}
The map \(\iota_X\) is indeed the desired bi-Lipschitz map after \autoref{prop: fine curve graph bi-lip}. Now, it suffices to show that the graph \(\CY\) is \(\delta\)-hyperbolic.

Let \(x,y,z\in \CY\) be arbitrary. We will show that there exists an \(\lambda>0\) such that any standard path between \(y,z\) is contained in a \(\lambda\)-neighbourhood of the union of any standard paths between \(x,y\) and \(x,z\).

Let \(\ell\) be a standard path between \(y\) and \(z\), and let \(W\in \Y_K(y,z)\cup\{\sigma(y),\sigma(z)\}\) be arbitrary such that \(\ell\) passes through \(\iota_W(\C^\dagger(W))\). Also consider \(\ell_1\) a standard path between \(x\) and \(y\), as well as \(\ell_2\) a standard path between \(x\) and \(z\). We recall from \cite{bowden2022quasi} that \(\C^\dagger(X)\) for any \(X\in \Y\) is \(\delta'\)-hyperbolic for some uniform \(\delta'>0\).

First, consider the case where \(d_W(x,y),d_W(x,z)>\Theta\). By our assumption, \(W\) is contained in \(\Y_\Theta(x,y)\) and \(\Y_\Theta(x,z)\). By \autoref{prop: path rigid}, there exists a \(v\in \ell_1\) such that
\[ d_{\CY}\big(v,\iota_W(\pi_W(x))\big),d_{\CY}\big(v,\iota_W(\pi_W(y))\big)<4L+4K\]
and together with \autoref{lem: coarse eq for curves} \(d^\pi_W(v,x), d^\pi_W(v,y)\pt K\). This implies that the \(\iota_W\)-image of the geodesic segment between \(\pi_W(x)\) and \(\pi_W(y)\) is within a distance \(\pt 4L+6K\), \emph{i.e.} the \(\iota_W\)-image of the geodesic segment between \(\pi_W(x)\) and \(\pi_W(y)\) is contained in a uniformly bounded neighbourhood of \(\ell_1\). The same arguments also imply that the \(\iota_W\)-image of the geodesic segment between \(\pi_W(x)\) and \(\pi_W(y)\) is contained in a uniformly bounded neighbourhood of \(\ell_2\). Since \(\C^\dagger(W)\) is uniformly hyperbolic, the \(\iota_W\)-image of the geodesic in \(\C^\dagger(W)\) between \(\pi_W(y),\pi_W(z)\) is within Hausdorff distance to those between \(\pi_W(x),\pi_W(y)\) and between \(\pi_W(x),\pi_W(z)\). But \(\iota_W\)-image of the geodesic in \(\C^\dagger(W)\) between \(\pi_W(y)\) and \(\pi_W(z)\) is also contained in a uniform neighbourhood of \(\ell\) after \autoref{lem: stand path proj closed} and \autoref{prop: fine curve graph bi-lip}. So we can conclude that \(\ell\cap \iota_W(\C^\dagger(W))\) is within a uniform neighbourhood of \(\ell_1\cup\ell_2\).

Suppose now \(d_W(x,y)\leq \Theta\). As \(W\in \Y_K(y,z)\cup\{\sigma(y),\sigma(z)\}\), we can conclude that \(d_W(x,z)>\Theta\) after \eqref{eq: coarse triangle ineq}, or \(\sigma(y),\sigma(z)\) are velcrot and \(d_W^\pi(y,z)\leq K\). The first situation can be concluded with the same arguments as above. The latter situation is trivial because \(\ell\) is contained in a uniform neighbourhood of \(y,z\in \ell_1\cup \ell_2\).

In conclusion, we have shown that \(\ell\) is contained in a uniform neighbourhood of \(\ell_1\cup \ell_2\). Since \(\ell_1,\ell_2,\ell\) are chosen independently and the constant for the uniform neighbourhood does not depend on the choice of \(\ell_i\)'s, we can conclude that \(\LL(y,z)\) is contained in a uniform neighbourhood of \(\LL(x,y)\cup \LL(x,z)\). Now, applying the Guessing Geodesics Lemma (\autoref{prop: ggl}), we can conclude that \(\CY\) is \(\delta\)-hyperbolic.
\end{proof}

\subsection{Application to quasi-morphisms} An important application of the blown-up fine projection complex is building quasi-morphisms on the group acting on this complex by isometries via the famous Bestvina--Fujiwara machinery.

\begin{proposition}\label{prop: ind loxo}
Let \(S,\Y,\C^\dagger(\Y),\Homeo_0(S)\) be as above. Then for each subsurface \(X\in \Y\), there exist \(f,g\in\Homeo_0(X;\partial X)< \Homeo_0(S)\) acting by independent loxodromic isometries on \(\C^\dagger(\Y)\). Moreover, the two elements \(f,g\) can be taken smooth.
\end{proposition}
\begin{proof}
By \cite{bowden2022quasi}, we can find \(f,g\in \Diff_0(X;\partial X)\subset \Homeo_0(S)\) that acts on \(\C^\dagger(X)\) by independent loxodromic isometries. We claim that via the \(\Homeo_0(S)\)-action on \(\C^\dagger(\Y)\), two elements \(f,g\) also act by independent loxodromic elements.

Since the embedding \(\iota_X\colon \C^\dagger(X)\hookrightarrow \CY\) is bi-Lipschitz, and \emph{a fortiori} quasi-isometric, by \autoref{prop: fine curve graph bi-lip}, we can conclude that \(f,g\) also yield loxodromic elements on \(\C^\dagger(\Y)\). Hence, to prove the claim, it suffices to show that they are independent.

Let \(A_f, A_g\subset \iota_X(\C^\dagger(X))\) be, respectively, quasi-axes of \(f\) and \(g\). Let \(B>0\) and \(h\in \Homeo_0(S)\) be arbitrary. Now, we fall into two possibilities:
\begin{mycases}
\item Suppose that \(hX\) is not velcrot to \(X\). Then \(hA_g\subset \iota_{hX}(\C^\dagger(hX))\). Let \(z\) be any vertex contained in \(\iota_{hX}(\C^\dagger(hX))\cap \mathcal{N}_B(A_f)\). By definition, there exists \(x=(\alpha,X)\in A_f\) such that \(d_{\CY}(x,z)\leq B\). Take \(x'\in \iota_{hX}(\C^\dagger(hX))\) to be one nearest point projection of \(x\) to the subspace \(\iota_{hX}(\C^\dagger(hX))\). By \autoref{cor: projection consistent}, there exists a uniform \(N>0\) such that
\[d_{\CY}\big(x',\iota_{hX}(\pi_{hX}(\alpha))\big)\leq N,\]
but as \(\alpha\cap \partial X=\emptyset\), we can further deduce from \autoref{prop_disjoint_subsurface_proj} that 
\[d_{\CY}\big(x',\iota_{hX}(\pi_{hX}(X))\big)\leq N,\]
for a sufficiently large \(N\) independent from the choice of \(B>0\) and \(h\in \Homeo_0(S)\). But \(d_{\CY}(x',z)\leq B-L\). Hence, we can conclude
\[d_{\CY}\big(z,\iota_{hX}(\pi_{hX}(X))\\big)\leq B-L-N,\]
which is uniformly bounded. Now, we need only to take a segment \(J\subset A_g\) such that the diameter is large than \(2(B-N-L)\) to see \(J\not\subset \mathcal{N}_B(A_f)\).
\item If \(hX\) is velcrot to \(X\), then by \autoref{prop: witness}, we can find a subsurface \(Z\subsetneqq X\cap hX\) that is isotopic to both \(X\) and \(hX\). Now take an isotopy \(h'\in \Homeo_0(S)\) such that \(h'|_Z=\Id_Z\) and \(h'hX=X\). In this way, we may consider \(h'h\in \Homeo_0(X)\). Note that every vertex on \(hA_g\) has essential subsurface projection on \(Z\), so do \(h'hA_g\). Moreover, as \(h'|_Z=\Id_Z\), we can conclude that for any \(x=(\alpha,X)\in A_g\), the subsurface projection \(\pi_Z(h\alpha)=\pi_Z(h'h\alpha)\), as \(h\alpha\cap Z=h'h\alpha\cap Z\). Therefore, by our construction of \(\CY\), we have
\[d_{\CY}(hx,h'hx)\leq 2L.\]
Note that \(h'h\in \Homeo_0(X)\) and \(f\nsim g\) for \(\Homeo_0(X)\)-action on \(\C^\dagger(X)\) and thus on \(\iota_X(\C^\dagger(X))\subset \CY\), there exists \(J\subset A_g\) such that \(h'hJ\not\subset \mathcal{N}_{B+2L}(A_f)\). This further yields \(hJ\not\subset \mathcal{N}_{B}(A_f)\).
\end{mycases}
Hence, we can finally conclude that \(f\nsim g\) for \(\Homeo_0(S)\)-action on \(\CY\).
\end{proof}

Now, we are able to conclude the non-sporadic cases in \autoref{thm: qm version}.
\begin{proof}[Proof of the non-sporadic cases in \autoref{thm: qm version}]
    Let \(\Y\) be a collection of essential subsurfaces satisfying \ref{Y1} and \ref{Y2} that contains the subsurface \(\Sigma\). Then the action of \(\Homeo_0(S)\) on \(\CY\) admits two independent loxodromic elements (\autoref{prop: ind loxo}). Hence, by \autoref{Bestvina--Fujiwara machinery}, we can construct homogeneous quasi-morphisms
    \[\varphi\colon \Homeo_0(S)\to \R\]
    such that \(\varphi\neq 0\) on the two smooth independent loxodromic elements.
    
    To make \(\varphi\) a \(C^0\)-continuous quasi-morphism, we will need the further modification. Note that the restriction \(\varphi|_{\Diff_0(S)}\) is also a homogeneous quasi-morphism on \(\Diff_0(S)<\Homeo_0(S)\) and, using Kotschick's automatic continuity arguments (see \cite{kotschick2008stable} for the original result and also \cite[Theorem~A.6]{bowden2022quasi} for the statement we use), it is \(C^0\)-continuous. By Whitney approximation theorem (see for example \cite[Theorem~6.26]{lee2002intro}), the subgroup \(\Diff_0(S)<\Homeo_0(S)\) is dense. While \(\Homeo_0(S)\) is metrisable, one can extend \(\varphi|_{\Diff_0(S)}\) to a \(C^0\)-continuous function, which will also be a homogeneous quasi-morphism on \(\Homeo_0(S)\).
\end{proof}
\section{Projection complex for once-bordered torus}\label{sec-6}

In this last part, we will briefly explain why and how our construction also works for essential subsurfaces that are homeomorphic to a once-bordered torus.

\subsection{Subsurface projection and velcrotness}

Let \(\Sigma\) be a once bordered torus. Its \emph{fine curve graph} is defined slightly differently compared to the other surfaces. Let us still denote by \(\C^\dagger(\Sigma)\) its fine curve graph. Its vertices are then essential curves on \(\Sigma\) and we attach an edge to two vertices if the associated curves \emph{intersect transversely at most once}. As in \cite[Section 5.2]{bowden2022quasi}, this graph is Gromov hyperbolic and of infinite diameter.

A similar distance estimate to \autoref{prop: intersection} also holds for this version of fine curve graph:
\begin{lemma}[Lemma 2.7, \cite{bowden2022rotation}]\label{lem: intersection}
Let \(\Sigma\) and \(\C^\dagger(\Sigma)\) be as above. Then for any transverse \(\alpha,\beta\in\C^\dagger(\Sigma)\), we have
\[d^\dagger_\Sigma(\alpha,\beta)\leq 2 |\alpha\cap \beta|+2.\]
\end{lemma}

Now let \(S\) be a closed orientable surface of at least \(2\) genus and \(\Sigma\subset S\) be an essential subsurface that is a once-bordered torus. Similarly, we say that \(\alpha\in \C^\dagger(S)\) intersects \(\Sigma\) essentially if \(\alpha\cap \Sigma\) contains an essential curve or arc on \(\Sigma\). In that case, if \(\alpha\subset \Sigma\), we can set \(\pi_\Sigma(\alpha)=\{\alpha\}\); otherwise, we can also define the subsurface projection \(\pi_\Sigma(\alpha)\subset \C^\dagger(\Sigma)\) by the collection of essential curves on \(\Sigma\) that can taken as a boundary component of a regular neighbourhood of \(\partial \Sigma\) and an essential arc in \(\alpha\cap \Sigma\). With \autoref{lem: intersection}, the same arguments from \cite{long2025connected} yield the following:

\begin{proposition}
Let \(\Sigma, \C^\dagger(\Sigma)\) be as above and let \(\alpha\in \C^\dagger(S)\) be a curve intersecting \(\Sigma\) essentially. Then we have
\[\diam_{\C^\dagger(\Sigma)}\big(\pi_\Sigma(\alpha)\big)\leq 7.\]
If both \(\alpha,\beta\in \C^\dagger(S)\) intersect \(\Sigma\) essentially and \(\alpha\cap \beta=\emptyset\), then we also have
\[\diam_{\C^\dagger(\Sigma)}\big(\pi_\Sigma(\alpha)\cup \pi_\Sigma(\beta)\big)\leq 7.\]
\end{proposition}

Let \(\Sigma_1\) and \(\Sigma_2\) be two essential subsurfaces of \(S\) that are once-bordered tori. We say that \(\Sigma_1\) intersects \(\Sigma_2\) \emph{essentially} if \(\partial \Sigma_1\cap \Sigma_2\) contains an essential arc on \(\Sigma_2\). Similarly, we say that \(\Sigma_1,\Sigma_2\) are \emph{velcrot} if the collection of common curves on \(\Sigma_1\) and \(\Sigma_2\) is unbounded in both \(\C^\dagger(\Sigma_1)\) and \(\C^\dagger(\Sigma_2)\). 

Many of the previous discussions can be concluded in a similar way, although some additional care should be taken for the cases of once-bordered tori. Here, we will include bare-bones proofs of some main results.

\begin{proposition}\label{prop: tori witness}
Let \(\Sigma_1\) and \(\Sigma_2\) be two essential subsurfaces of \(S\) that are once-bordered tori. Then \(\Sigma_1\) and \(\Sigma_2\) are velcrot if and only if there exists \(\Sigma\subset \Sigma_1\cap \Sigma_2\) isotopic to both \(\Sigma_i\) for \(i=1,2\).
\end{proposition}
\begin{proof}
For the ``if'' part, following the arguments from \autoref{lem: annul homo velcrot}, we can conclude that the natural inclusions \(\C^\dagger(\Sigma)\hookrightarrow\C^\dagger(\Sigma_i)\) for \(i=1,2\) are isometric embedding, as we also have the distance formula as in \autoref{lem: dagger} for once-bordered tori after the comments from \cite[Section 5.2]{bowden2022quasi}. For the ``only if'' part, we apply the same arguments as in \autoref{prop: witness}.
\end{proof}

We similarly define the projection distance as in \autoref{sec: surf surg}. The proof of the following corollary goes \emph{verbatim} as \autoref{cor: velcrot same base}:
\begin{corollary}\label{cor: velcrot same base tori}
There exists \(M>0\) that verifies the following. Let \(\Sigma_1,\Sigma_2\) be two essential velcrot once-bordered tori on \(S\). Then 
\[|d^\pi_{\Sigma_1}(x,z)-d^\pi_{\Sigma}(x,z)|<M,\]
for any \(x,z\in \C^\dagger(S)\) that intersect \(\Sigma_1,\Sigma_2\) essentially. \QEDD
\end{corollary}

Now, using the same arguments from \autoref{sec: surf surg}, we can also conclude the following results:
\begin{theorem}[Fine Behrstock's inequality for once-bordered tori]\label{thm: berhstock for torus}
There exists \(M>0\) such that the following holds. Let \(\Sigma_1,\Sigma_2,\Sigma_3\) be three essential subsurfaces of \(S\) that are once-bordered tori. Suppose that \(\Sigma_i\) and \(\Sigma_j\) \((i\neq j)\) are either overlapping or isotopic. If \(d^\pi_{\Sigma_1}(\Sigma_2,\Sigma_3)>M\), then \(d^\pi_{\Sigma_2}(\Sigma_1,\Sigma_3),d^\pi_{\Sigma_3}(\Sigma_1,\Sigma_2)<M\). \QEDD
\end{theorem}

\begin{proposition}\label{prop: finiteness for torus}
There exists a constant \(M>0\), such that the following holds. For any two curves \(\alpha,\beta\in \C^\dagger(S)\), there are finitely many essential subsurfaces \((\Sigma_i)_{i=1}^n\) of \(S\) that are once-bordered tori such that if \(Z\) is an essential subsurface of \(S\) that is also a once-bordered torus, with \(M<d^\pi_Z(\alpha,\beta)<\infty\), then \(Z\) is velcrot to one of \(\Sigma_i\)'s. In particular, these \(\Sigma_i\)'s can be made pairwise non-velcrot with \(M<d^\pi_{\Sigma_i}(\alpha,\beta)<\infty\). \QEDD
\end{proposition}

Finally, we remark that for \autoref{thm: berhstock main}, the arguments are the same as \autoref{thm: berhstock} and it suffices to apply the results for tori or non-sporadic surfaces accordingly.

\subsection{Projection complexes and quasi-morphisms}
By now, we have prepared all the necessary prerequisites for running the machinery established in \autoref{sec-4} and \autoref{sec-5}.

To be precise, we can pick \(\Y\) to be the collection of all essential subsurfaces on \(S\) that are homeomorphic to a once-bordered torus. Then they also enjoy the properties in \autoref{thm: BBF prerequi}, which is sufficient to construct a quasi-tree out of \(\Y\), also denoted by \(\PY\), where \(K>0\) is a sufficiently large number. This space is indeed unbounded, as we can apply the same arguments from \autoref{subsec: puncturing} for their surviving curve graphs.

Now, we build the blown-up projection complex in the same way as before, \emph{cf.} \ref{L1}, \ref{L2} and \ref{L3} in \autoref{sec-5A}. This gives us a \(\delta\)-hyperbolic space, still denoted by \(\CY\), on which \(\Homeo_0(S)\) acts by isometries.

Recall that for a once-bordered torus \(\Sigma\), if we consider the action of \(\Homeo_0(\Sigma)\) on the associated fine curve graph \(\C^\dagger(\Sigma)\), it also has two elements in \(\Diff_0(\Sigma)\subset \Homeo_0(\Sigma)\) that act as independent loxodromic isometries on \(\C^\dagger(\Sigma)\), see \cite[\S 5.2]{bowden2022quasi}. Using the same arguments as before, we can also conclude:
\begin{proposition}
Let \(S\) be a closed surface of genus at least \(2\) and \(\Y\) be the collection of all essential subsurfaces on \(S\) that are homeomorphic to a once-bordered torus. Let \(\CY,\Homeo_0(S)\) be as above. Then for each subsurface \(\Sigma\in \Y\), there exist \(f,g\in \Homeo_0(S)\) such that \(f(\Sigma)=g(\Sigma)=\Sigma\) acting by independent loxodromic isometries on \(\CY\). \QEDD
\end{proposition}

Now, using \autoref{Bestvina--Fujiwara machinery} and Kotschick's automatic continuity arguments, we can also conclude the cases for once-bordered tori in \autoref{thm: main}, \emph{i.e.} if \(\Sigma\subset S\) is an essential subsurface homeomorphic to a once-bordered torus, then there exists a \(C^0\)-continuous homogeneous quasi-morphisms \(\varphi\) on \(\Homeo_0(S)\) with \(\varphi(g)\neq 0\) for some \(g\in \Homeo_0(\Sigma;\partial \Sigma)<\Homeo_0(S)\).
\section{Asymptotic dimension}

In this last section, we will discuss some remarks on the asymptotic dimension of fine curve graphs and blown-up fine projection complexes. 

Let \(X\) be a metric space. We will denote by \(\asdim(X)\) its asymptotic dimension. This quantity takes integral values and can be defined for general metric spaces. Also, it is a quasi-isometric invariance, see, for example, \cite{bell2008asymptotic}.

For a geodesic Gromov hyperbolic space \(X\), its asymptotic dimension is closely related to the topological dimension of its Gromov boundary \(\partial X\). The following result is due to \cite[Proposition~6.2]{buyalo2008dimensions} for proper cases and is generalised to non-proper situations in \cite[Proposition~2.5]{kopreski2025asymp}:
\begin{proposition}\label{prop: bnd dim}
Let \(X\) be a geodesic Gromov hyperbolic space with a compact subset \(Z\subset \partial X\). Then \(\asdim (X)\geq \dim(Z)+1\).
\end{proposition}

To conclude \autoref{thm: asdim}, we now have reduced the problem to finding subspaces in \(\partial\C^\dagger(X)\) with arbitrarily large topological dimension. We remark that topological dimension is a homeomorphism invariance, see for example \cite[\S1.6]{engelking1978dimension}. So, it suffices to embed spaces of arbitrarily large topological dimension into \(\partial\C^\dagger(X)\). To this end, we recall the following topological interpretation of some boundary points on \(\partial \C^\dagger(X)\):
\begin{theorem}[Theorem~1.1, \cite{bowden2024boundary}]\label{thm: bnd}
Suppose that \(F\subset X\) is a finite collection of points, \(\rho\) a complete hyperbolic metric of finite area on \(X\setminus F\), and \(\lambda\subset X\setminus F\) a minimal geodesic lamination, which is not disjoint from any essential simple closed curve on \(X\setminus F\). Then \(\lambda\) determines a unique boundary point \(\xi_\lambda\in\partial \C^\dagger(X)\). Moreover, the stabiliser of \(\xi_\lambda\) consists exactly those homeomorphisms of \(X\) that preserves the lamination \(\lambda\) as a subset on \(X\).
\end{theorem}

\begin{remark}
Although \autoref{thm: bnd} is originally stated for closed surfaces, for surfaces with boundary components, the same arguments apply and we can conclude the same result for the geometric interpretation of boundary points of the associated fine curve graph.
\end{remark}

The following construction is communicated to us by Frédéric Le Roux:
\begin{proposition}\label{prop: embed}
Let \(S\) be a compact connected oriented non-sporadic surface, possibly with boundary components. Then there is a homeomorphic embedding \([0,1]^n\hookrightarrow\partial \C^\dagger(S)\), for any positive integer \(n>0\).
\end{proposition}
\begin{proof}
Let us fix a geodesic lamination \(\lambda\subset X\setminus F\) and let \(\xi_\lambda\in \C^\dagger(X)\) be its associated boundary point as in \autoref{thm: bnd}. Now, take \(n\) disjoint discs \(D_1,\dots,D_n\subset S\) with \(D_i\cap \lambda\neq \emptyset\). For each \(D_i\), consider a homeomorphism \(f_i\in \Homeo_0(S)\) such that \(\mathrm{supp}(f_i)\subset \inte(D_i)\) and that \(f_i(\lambda)\neq \lambda\). To be precise, up to a local homeomorphism, \(D_i\) can be taken as a unit disc in \(\R^2\) with \(\lambda\cap D_i\) identified with horizontal segments, so \(f_i\) can be taken as a Dehn twist along a simple closed curve \(\alpha\) in \(D_i\) such that \(\alpha\) is transverse to some leaves in \(\lambda\). Then Alexander's trick yields an isotopy \(f_i^t\in \Homeo_0(X)\) for \(t\in [0,1]\) such that \(\mathrm{supp}(f^t_i)\subset \inte(D_i)\) and \(f^t_i(\lambda)\neq f^s_i(\lambda)\neq \lambda\) for all \(0<s\neq t\leq 1\) with \(f^1_i= f_i\) and \(f_i^0=\Id\). As \(f^t_i\) and \(f^s_j\) have disjoint support for any \(0< s,t\leq 1\) and any \(i\neq j\), these two elements commute. Therefore, it yields a continuous embedding of \(p\colon [0,1]^n\hookrightarrow \Homeo_0(X) \) by
\[p\big((t_i)_{1\leq i\leq n}\big)\coloneqq \prod_{i=1}^n f_i^{t_i}\in \Homeo_0(S).\]
Note that \(p(x)(\lambda)\neq p(y)(\lambda)\) whenever \(x\neq y\). As \(\Homeo_0(X)\) acts on \(\partial \C^\dagger(X)\) continuously \cite[\S6]{long2025connected}, it turns out that by \([0,1]^n\ni x\mapsto p(x)(\lambda)\), following \autoref{thm: bnd}, we have built a continuous and injective map \([0,1]^n\hookrightarrow \partial\C^\dagger(X)\). As \([0,1]^n\) is a compact space and \(\partial\C^\dagger(X)\) is Hausdorff, this map is a homeomorphism upon its image.
\end{proof}

With Le Roux's construction, we can now conclude \autoref{thm: asdim}.

\begin{proof}[Proof of \autoref{thm: asdim}]
Let \(Z_n\subset \partial\C^\dagger(S)\) be the image of \([0,1]^n\) under the homeomorphic embedding constructed in \autoref{prop: embed}. As \(\dim(Z_n)=n\), by \autoref{prop: bnd dim}, we can conclude that \(\asdim\big(\C^\dagger(S)\big)\geq n+1\). But \(n>0\) can be chosen arbitrarily large, so \(\asdim\big(\C^\dagger(S)\big)=\infty\).
\end{proof}

Unlike the construction in \cite{bestvina2015constructing}, as \(\C^\dagger(X)\) is quasi-isometrically embedded in \(\CY\) (\autoref{prop: fine curve graph bi-lip}), we thus have:
\begin{corollary}\label{cor: CY infinite asdim}
The graph \(\CY\) constructed above has infinite asymptotic dimension. \QEDD
\end{corollary}

\bibliographystyle{alphaurl}
\bibliography{references}

\noindent{\sc Yongsheng JIA}\\
\noindent{\sc Manchester University, Department of Mathematics, M13 9PL, Manchester, UK}

\noindent{\it Email address:} {\tt \href{mailto:manchesterjia1999@gmail.com}{manchesterjia1999@gmail.com}}

\vspace{1cm}

\noindent{\sc Yusen LONG}\\
\noindent{\sc Université Paris-Est Créteil, CNRS, LAMA UMR8050, F-94010 Créteil, France}

\noindent{\it Email address:} {\tt \href{mailto:yusen.long@u-pec.fr}{yusen.long@u-pec.fr}}

\end{document}